\documentclass[reqno, 12pt]{amsart}
\usepackage{amsfonts}
\usepackage{mathrsfs}
\usepackage{amssymb,url}
\usepackage{graphicx}

\date{}

\allowdisplaybreaks[4]
\theoremstyle{plain}
\newtheorem{thm}{Theorem}[section]

\newtheorem{lem}[thm]{Lemma}

\newtheorem{rem}{Remark}

\theoremstyle{definition}

\usepackage[colorlinks,linkcolor=blue,urlcolor=red,linktocpage]{hyperref}

\numberwithin{equation}{section}

\begin{document}

\title
[~~~]{ Origin-Symmetric Bodies of Revolution with Minimal Mahler
Volume in $\mathbb{R}^3$-a new proof}

\author[Youjiang Lin]{Youjiang Lin}
\address{School of Mathematical Sciences, Peking University, Beijing, People's Republic of China 100871}
 \email{\href{mailto: YOUJIANG LIN
<lxyoujiang@126.com>}{lxyoujiang@126.com}}
\author[Gangsong Leng]{Gangsong Leng}
\address{Department of Mathematics, Shanghai University, Shanghai, People's Republic of China 200444} \email{\href{mailto:
Gangsong Leng <gleng@staff.shu.edu.cn>}{gleng@staff.shu.edu.cn} }

\begin{abstract} In \cite{Me98}, Meyer and Reisner proved the Mahler conjecture for rovelution bodies.  In
this paper, using a new method, we prove that among \textit{origin-symmetric bodies of
revolution} in $\mathbb{R}^3$, cylinders have the minimal Mahler
volume. Further, we prove that among \textit{parallel sections
homothety bodies} in $\mathbb{R}^3$, 3-cubes have the minimal Mahler
volume.

\end{abstract}



\thanks{The authors would like to acknowledge the
support from the 973 Program 2013CB834201, National Natural Science
Foundation of China under grant 11271244. }

\maketitle

\section{Introduction}

The well-known  Mahler's conjecture (see, e.g.,\cite{Ga06},
\cite{Ma39}, \cite{Sc93} for references) states that, for any
origin-symmetric convex body $K$ in $\mathbb{R}^n$,
\begin{eqnarray}
\mathcal{P}(K)\geq \mathcal{P}(C^n)=\frac{4^n}{n!} ,
\end{eqnarray}
where $C^n$ is an $n$-cube and $\mathcal{P}(K)=Vol(K)Vol(K^{\ast})$,
which is known as the {\it Mahler volume} of $K$.

For $n=2$, Mahler \cite{Ma39b} himself proved the conjecture, and in
1986 Reisner \cite{Re86} showed that equality holds only for
parallelograms.  For $n=2$, a new proof of inequality (1.1) was
obtained by Campi and Gronchi \cite{CG06}. Recently, Lin and Leng
\cite{LL10} gave a new and intuitive proof of the inequality (1.1)
in $\mathbb{R}^2$.

For some special classes of origin-symmetric convex bodies in
$\mathbb{R}^n$, a sharper estimate for the lower bound of $\mathcal
{P}(K)$ has been obtained. If $K$ is a convex body which is
symmetric around all coordinate hyperplanes, Saint Raymond
\cite{SR81} proved that $\mathcal {P}(K)\geq 4^n/n!$; the equality
case was discussed in \cite{Me86, Re87}. When $K$ is a zonoid
(limits of finite Minkowski sums of line segments), Meyer and
Reisner (see, e.g., \cite{GMR88, Re85, Re86}) proved that the same
inequality holds, with equality if and only if $K$ is an $n$-cube.
For the case of polytopes with at most $2n+2$ vertices (or facets)
(see, e.g., \cite{Ba95} for references), Lopez and Reisner
\cite{LR98} proved the inequality (1.1) for $n\leq 8$ and the
minimal bodies are characterized. Recently, Nazarov, Petrov,
Ryabogin and Zvavitch \cite{NPRZ09} proved that the cube is a strict
local minimizer for the Mahler volume in the class of
origin-symmetric convex bodies endowed with the Banach-Mazur
distance.

Bourgain and Milman \cite{BM87} proved that there exists
 a universal constant $c>0$ such that $\mathcal {P}(K)\geq c^n \mathcal{P}(B)$,
which is now known as the reverse Santal\'{o} inequality. Very
recently, Kuperberg \cite{Ku08} found a beautiful new approach to
the reverse Santal\'{o} inequality. What's especially remarkable
about Kuperberg's inequality is that it provides an explicit value
for $c$.

Another variant of the Mahler conjecture without the assumption of
origin-symmetry states that, for any convex body $K$ in
$\mathbb{R}^n$,
\begin{eqnarray}
\mathcal{P}(K)\geq \frac{(n+1)^{(n+1)}}{(n!)^2},
\end{eqnarray}
with equality conjectured to hold only for simplices. For $n=2$,
Mahler himself proved this inequality in 1939 (see, e.g.,\cite{Ca06,
Cam06, LYZ10} for references) and Meyer \cite{Me91} obtained the
equality conditions in 1991. Recently, Meyer and Reisner\cite{MR06}
have proved inequality (1.2) for polytopes with at most $n+3$
vertices. Very recently, Kim and Reisner\cite{Ki} proved that the
simplex is a strict local minimum for the Mahler volume in the
Banach-Mazur space of $n$-dimensional convex bodies.

Strong functional versions of the Blaschke-Santal\'{o} inequality
and its reverse form have been studied recently (see, e.g.,
\cite{Ar04, Fr10, Fr07, Fr08, Fra08,Me98} ).

The Mahler conjecture is still open even in the three-dimensional
case. Terence Tao in \cite{Ta07} made an excellent remark about the
open question.

To state our results, we first give some definitions. In the
coordinate plane XOY of $\mathbb{R}^3$, let
\begin{eqnarray}
D=\{(x,y): -a\leq x\leq a, |y|\leq f(x)\},
\end{eqnarray}
where  $f(x)$ ($[-a,a]$, $a>0$) is a concave, even and nonnegative
function. An {\it origin-symmetric body of revolution} $R$ is
defined as the convex body generated by rotating $D$ around the
$X$-axis in $\mathbb{R}^3$. $f(x)$ is called its {\it generating
function} and $D$ is its {\it generating domain}. If the generating
domain of $R$ is a rectangle (the generating function of $R$ is a
constant function),  $R$ is called a {\it cylinder}. If the
generating domain of $R$ is a diamond (the generating function
$f(x)$ of $R$ is a linear function on $[-a,0]$ and $f(-a)=0$), $R$
is called a {\it bicone}.


In this paper, we prove that cylinders have the minimal Mahler
volume for  origin-symmetric bodies of revolution in $\mathbb{R}^3$.

\begin{thm}For any origin-symmetric body of revolution
$K$ in $\mathbb{R}^3$, we have
\begin{eqnarray}
\mathcal {P}(K)\geq \frac{4\pi^2}{3},
\end{eqnarray}
and the equality holds if and only if $K$ is a cylinder or bicone.
\end{thm}
\begin{rem}
 In \cite{Me98}, for the Schwarz rounding $\tilde{K}$ of
a convex body $K$ in $\mathbb{R}^n$, Meyer and Reisner gave a lower
bound for $\mathcal {P}(\tilde{K})$. Especially, for a general body
of revolution $K$ in $\mathbb{R}^3$, they proved
\begin{eqnarray}
\mathcal {P}(K)\geq \frac{4^4\pi^2}{3^5},
\end{eqnarray}
with equality if and only if $K$ is a cone and $|AO|/|AD|=3/4$
{\rm(}where, $A$ is the vertex of the cone and $AD$ is the height
and $O$ is the Santal\'{o} point of $K${\rm)}.
\end{rem}

The following Theorem $1.2$ is the functional version of the Theorem
1.1.
\begin{thm} Let $f(x)$  be a concave, even and nonnegative
function defined on $[-a,a]$, $a>0$, and for $x^{\prime}\in
[-\frac{1}{a},\frac{1}{a}]$ define
$$f^{\ast}(x^{\prime})=\inf_{x\in[-a,a]}\frac{1-x^{\prime}x}{f(x)}.$$
Then, we have
\begin{eqnarray}
\left(\int_{-a}^{a} (f(x))^2 dx\right)\left(
\int_{-\frac{1}{a}}^{\frac{1}{a}} (f^{\ast}(x^{\prime}))^2
dx^{\prime}\right)\geq\frac{4}{3},
\end{eqnarray}
with equality if and if $f(x)=f(0)$ or
$f^{\ast}(x^{\prime})=1/f(0)$.
\end{thm}

Let $C$ be an origin-symmetric convex body in the  coordinate plane
YOZ of $\mathbb{R}^3$ and $f(x)$ ($x\in [-a,a]$, $a>0$) is a
concave, even and nonnegative function. A {\it parallel sections
homothety body} is defined as the convex body
$$K=\bigcup_{x\in[-a,a]}\{f(x)C+xv\},$$
where $v=(1,0,0)$ is a unit vector in the positive direction of the
X-axis, $f(x)$ is called its {\it generating function} and $C$ is
its {\it homothetic section}.


Applying Theorem 1.2, we prove that among parallel sections
homothety bodies in $\mathbb{R}^3$, 3-cubes have the minimal Mahler
volume.

\begin{thm}For any parallel sections
homothety body $K$ in $\mathbb{R}^3$, we have
\begin{eqnarray}
\mathcal {P}(K)\geq \frac{4^3}{3!},
\end{eqnarray}
and the equality holds if and only if $K$ is a 3-cube or octahedron.
\end{thm}

\section{ Definitions, notation, and preliminaries}
As usual, $S^{n-1}$ denotes the unit sphere, and $B^n$ the unit ball
centered at the origin, $O$ the origin and $\|\cdot\|$ the norm in
Euclidean $n$-space $\mathbb{R}^n$. The symbol for the set of all
natural numbers is $\mathbb{N}$. Let $\mathcal {K}^n$ denote the set
of convex bodies (compact, convex subsets with non-empty interiors)
in $\mathbb{R}^n$. Let $\mathcal {K}^n_o$ denote the subset of
$\mathcal {K}^n$ that contains the origin in its interior. For $u\in
S^{n-1}$, we denote by $u^{\perp}$ the $(n-1)$-dimensional subspace
orthogonal to $u$. For $x$, $y\in \mathbb{R}^n$, $ x\cdot y$ denotes
the inner product of $x$ and $y$.

Let $\textrm{int}\;K$ denote the interior of $K$. Let
$\textrm{conv}\;K$ denote the convex hull of $K$. we denote by
$V(K)$ the $n$-dimensional volume of $K$. The notation for the usual
orthogonal projection of $K$ on a subspace $S$ is $K|S$.

If $K\in {K}^n_o$, we define the {\it polar body} $K^{\ast}$ of $K$
by
$$K^{\ast}=\{ x\in \mathbb{R}^n:~ x\cdot y \leq 1~, \forall y\in K\}.$$
\begin{rem}
If $P$ is a polytope, i.e., $P={\rm conv}\{p_1,\cdots,p_m\}$, where
$p_i$ $(i=1,\cdots,m)$ are vertices of polytope $P$. By the
definition of the polar body, we have
\begin{eqnarray}
P^{\ast}&=&\{x\in \mathbb{R}^n: x\cdot p_1\leq 1,\cdots,  x\cdot
p_m\leq 1\}
\nonumber\\
&=&\bigcap_{i=1}^{m}\{x\in \mathbb{R}^n: x\cdot p_i \leq 1\},
\end{eqnarray}
which implies that $P^{\ast}$ is an intersection of $m$ closed
halfspaces with exterior normal vectors $p_i$
{\rm(}$i=1,\cdots,m${\rm)} and the distance of hyperplane $$\{x\in
\mathbb{R}^n: x\cdot p_i= 1\}$$ from
 the origin is $1/\|p_i\|$.
 \end{rem}

Associated with each convex body $K$ in $\mathbb{R}^n$ is its  {\it
support function} $h_K: \mathbb{R}^n\rightarrow [0,\infty)$, defined
for $x\in \mathbb{R}^n$, by
\begin{eqnarray}
h_K(x)= \max\{y\cdot x: y\in K\},
\end{eqnarray}
and its  {\it radial function} $\rho_K:
\mathbb{R}^n\backslash\{0\}\rightarrow (0,\infty)$, defined for
$x\neq 0$, by
\begin{eqnarray}
\rho_K(x)= \max\{\lambda\geq 0~:~\lambda x\in K\}.
\end{eqnarray}
%
%
%

For $K$, $L\in \mathcal{K}^n$, the {\it Hausdorff distance} is
defined by
\begin{eqnarray}
\delta(K,L)=\min\{\lambda\geq0:~K\subset L+\lambda B^n,~L\subset
K+\lambda B^n\}.
\end{eqnarray}

A  {\it linear transformation} (or  {\it affine transformation}) of
$\mathbb{R}^n$ is a map $\phi$ from $\mathbb{R}^n$ to itself such
that $\phi x~=~A x$ (or $\phi x~=~A x + t$, respectively), where $A$
is an $n \times n$ matrix and $t\in \mathbb{R}^n$. It is known that
Mahler volume of $K$ is invariant under affine transformation.

 For $K\in \mathcal {K}^n_o$, if
$(x_1,x_2,\cdots,x_n)\in K$, we have
$(\varepsilon_1x_1,\cdots,\varepsilon_nx_n)\in K$ for any signs
$\varepsilon_i=\pm1$ ($i=1,\cdots,n$), then $K$ is a {\it
1-unconditional convex body}. In fact, $K$ is symmetric with respect
to all coordinate planes.

The following Lemma 2.1 will be used to calculate the volume of an
origin-symmetric body of revolution. Since the lemma is an
elementary conclusion in calculus, we omit its proof.
\begin{lem}
In the coordinate plane XOY,  let
$$D=\{(x,y): a\leq x\leq b, |y|\leq f(x)\},$$
where $f(x)$ is a linear, nonnegative function defined on $[a,b]$.
Let $R$ be a body of revolution generated by $D$. Then
\begin{eqnarray}
V(R)=\frac{\pi}{3}(b-a)\left[f(a)^2+f(a)f(b)+f(b)^2\right].
\end{eqnarray}

\end{lem}

\section{Main result and its proof}

In the paper, we consider convex bodies in a three-dimensional
Cartesian coordinate
 system with origin $O$ and
 its three coordinate axes are denoted by $X$-axis, $Y$-axis,
  and $Z$-axis.

\begin{lem} If $K\in \mathcal {K}_0^3$,
then for any $u\in S^2$, we have
\begin{eqnarray}
K^{\ast}\cap u^{\perp}=(K|u^{\perp})^{\ast}.
\end{eqnarray}
On the other hand, if $K^{\prime}\in\mathcal {K}_0^3$ satisfies
\begin{eqnarray}
K^{\prime}\cap u^{\perp}=(K|u^{\perp})^{\ast}
\end{eqnarray}
for any $u\in S^2\cap v_0^{\perp}$ {\rm (}$v_0$ is a fixed
vector{\rm )}, then,
\begin{eqnarray}
K^{\prime}=K^{\ast}.
\end{eqnarray}
\end{lem}
\begin{proof}

Firstly, we prove (3.1).

 Let $x\in u^{\perp}$, $y\in K$ and
$y^{\prime}=y|u^{\perp}$, since the hyperplane $u^{\perp}$ is
orthogonal to the vector $y-y^{\prime}$,  then $$ y\cdot x=
(y^{\prime}+y-y^{\prime})\cdot x=y^{\prime}\cdot x+
(y-y^{\prime})\cdot x= y^{\prime}\cdot x.$$

 If $x\in K^{\ast}\cap u^{\perp}$, for any $y^{\prime}\in
K|u^{\perp}$, there exists $y\in K$ such that
$y^{\prime}=y|u^{\perp}$, then $x\cdot y^{\prime}=x\cdot y\leq 1$,
thus $x\in (K|u^{\perp})^{\ast}$. Thus, we have $K^{\ast}\cap
u^{\perp}\subset (K|u^{\perp})^{\ast}$.

 If $x\in (K|u^{\perp})^{\ast}$,
then for any $y\in K$ and $y^{\prime}=y|u^{\perp}$, $ x\cdot
y=x\cdot y^{\prime}\leq 1$, thus $x\in K^{\ast}$, and since $x\in
u^{\perp}$, thus $x\in K^{\ast}\cap u^{\perp}$. Thus, we have
$(K|u^{\perp})^{\ast}\subset K^{\ast}\cap u^{\perp}$.

Next we prove (3.3).

 Let $S^1= S^2\cap v_0^{\perp}$. For any vector $v\in S^2$, there exists a $u\in S^1$
satisfying $v\in u^{\perp}$. Since $K^{\prime}\cap
u^{\perp}=(K|u^{\perp})^{\ast}$ and $K^{\ast}\cap
u^{\perp}=(K|u^{\perp})^{\ast}$, thus $K^{\prime}\cap
u^{\perp}=K^{\ast}\cap u^{\perp}$. Hence, we have
$\rho_{K^{\prime}}(v)=\rho_{K^{\ast}}(v)$.  Since $v\in S^2$ is
arbitrary, we get $K^{\prime}=K^{\ast}.$
\end{proof}
\begin{lem} In the coordinate plane XOY, let $P$
be a 1-unconditional convex body. Let $R$ and $R^{\prime}$ be two
origin-symmetric bodies of revolution generated by $P$ and
$P^{\ast}$, respectively. Then $R^{\prime}=R^{\ast}$.
\end{lem}
\begin{proof}
Let $v_0=\{1,0,0\}$ and $S^1=S^2\cap v_0^{\perp}$, for any $u\in
S^1,$ we have $R|u^{\perp}=R\cap u^{\perp}$. Since $R^{\prime}\cap
u^{\perp}=P^{\ast}=(R\cap u^{\perp})^{\ast}$ for any $u\in S^1$,
thus $R^{\prime}\cap u^{\perp}=(R| u^{\perp})^{\ast}$ for any $u\in
S^1$. By Lemma 3.1, we have $R^{\prime}=R^{\ast}$.
\end{proof}

\begin{lem} For any origin-symmetric body of revolution $R$, there exists  a linear
transformation $\phi$ satisfying

(i) $\phi R$ is an origin-symmetric body of revolution;

(ii) $\phi R\subset C^3=[-1,1]^3,$ where $C^3$ is the unit cube in
$\mathbb{R}^3$.
\end{lem}
\begin{proof}
Let $f(x)$ ($x\in [-a,a]$) be the generating function of $R$.

For vector $v=(1,0,0)$ and any $t\in [-a,a],$ the set $R\cap
(v^{\perp}+tv)$ is a disk in the plane $v^{\perp}+tv$ with the point
$(t,0,0)$ as the center and $f(t)$ as the radius.

Next, for a $3\times 3$ diagonal matrix $A=\begin{bmatrix}
b&0&0\\
0&c&0\\
0&0&c
\end{bmatrix},$ where $b, c\in \mathbb{R}^{+}$, let $\phi R=\{Ax:x\in R\}$, we prove
that $\phi R$ is still an origin-symmetric body of revolution.

For $t^{\prime}\in [-ab,ab]$, if
$(t^{\prime},y^{\prime},z^{\prime})\in \phi R\cap
(v^{\perp}+t^{\prime}v)$, there is $(t,y,z)\in R\cap(v^{\perp}+tv)$
satisfying $t^{\prime}=bt,\;\;y^{\prime}=cy,\;\;z^{\prime}=cz.$
Hence, we have
$$\|(t^{\prime},y^{\prime},z^{\prime})-(t^{\prime},0,0)\|=c\|(t,y,z)-(t,0,0)\|\leq
c f(t),$$
which implies that $\phi R\cap (v^{\perp}+t^{\prime}v)\subset
B^{\prime}$, where $B^{\prime}$ is a disk in the plane
$v^{\perp}+t^{\prime}v$ with $(t^{\prime},0,0)$ as the center and
$cf(t^{\prime}/b)$ as the radius.

On the other hand, if $(t^{\prime}, y^{\prime}, z^{\prime})\in
B^{\prime}$, then
$\|(t^{\prime},y^{\prime},z^{\prime})-(t^{\prime},0,0)\|\leq c
f(t^{\prime}/b)$.
Let $t=t^{\prime}/b$, $y=y^{\prime}/c$ and $z=z^{\prime}/c$. Noting
$t^\prime\in [-ab,ab]$, we have $t\in [-a,a]$ and
$$\|(t,y,z)-(t,0,0)\|=\frac{1}{c}\|(t^{\prime},y^{\prime},z^{\prime})-(t^{\prime},0,0)\|\leq
f(t).$$
Hence, we have $(t,y,z)\in R\cap (v^{\perp}+tv)$, which implies that
$(t^{\prime},y^{\prime},z^{\prime})=(bt,cy,cz)\in \phi
R\cap(v^{\perp}+t^{\prime}v)$. Thus, $B^{\prime}\subset \phi R\cap
(v^{\perp}+t^{\prime}v)$. Therefore, we have $ \phi R\cap
(v^{\perp}+t^{\prime}v)= B^{\prime}$. It follows that $\phi R$ is an
origin-symmetric body of revolution and its generating function is
$F(x)=cf(x/b)$, $x\in [-ab,ab]$.

 Set
$b=1/a$ and $c=1/f(0)$, we obtain $\phi R\subset C^3=[-1,1]^3$.

\end{proof}
\begin{rem}
By Lemma 3.3 and the affine invariance of Mahler volume, to prove
our theorems, we need only consider the origin-symmetric body of
revolution $R$ whose generating domain $P$ satisfies $T\subset
P\subset Q$, where
$$T=\{(x,y): |x|+|y|\leq 1\}\;\;\textrm{and}\;\;Q=\{(x,y): \max\{|x|,|y|\}\leq 1\}.$$
\end{rem}

In the following lemmas, let $\triangle ABD$ denote
$\textrm{conv}\{A, B, D\}$, where $A=(-1,1)$, $B=(0,1)$ and
$D=(-1,0)$.

\begin{figure}[htb]
\centering
 \includegraphics[height=10cm]{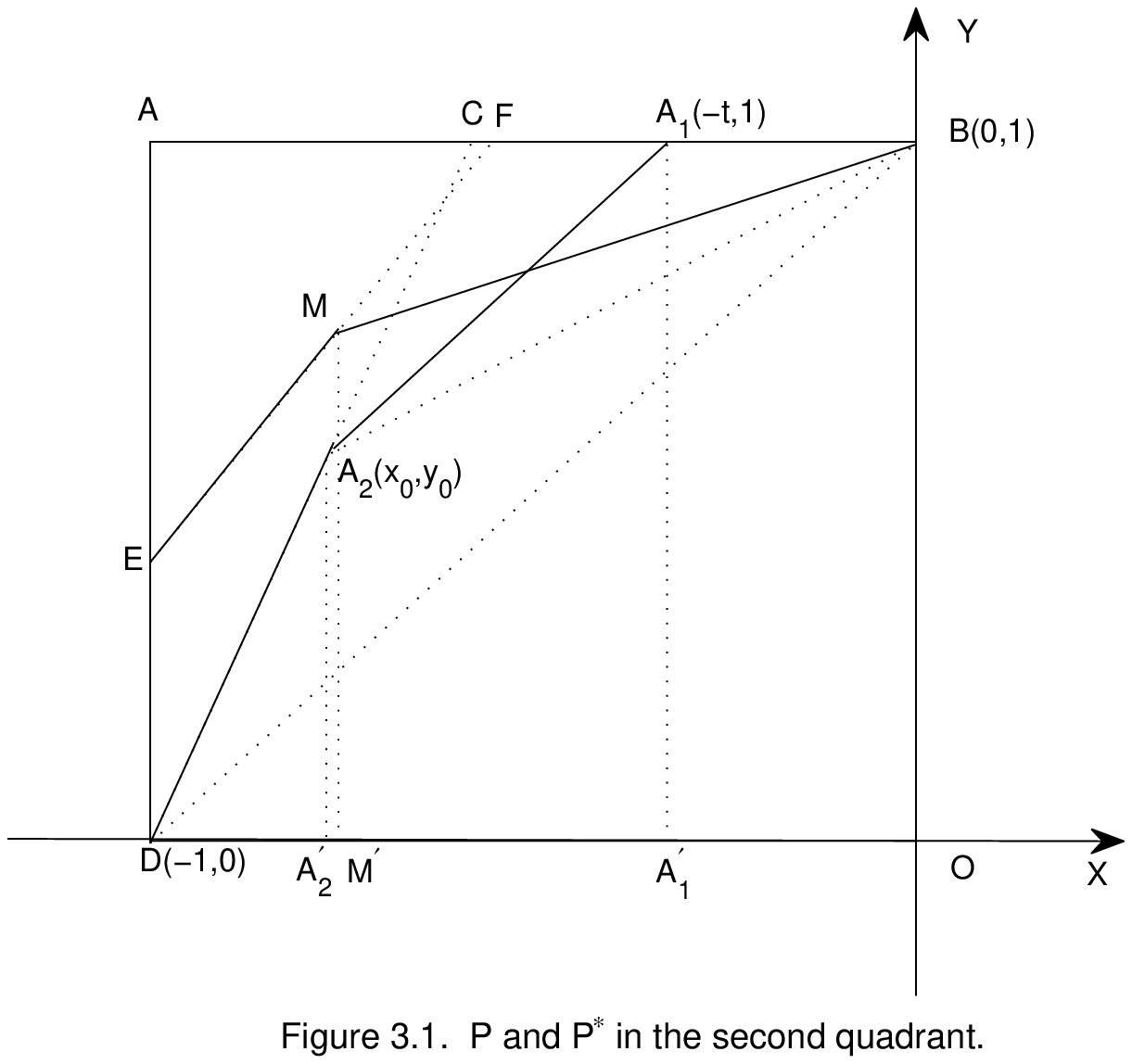 }
\end{figure}
\begin{lem} Let $P$ be a 1-unconditional polygon in the coordinate plane $XOY$ satisfying
$$P\cap \{(x,y): x\leq 0,\;y\geq 0\}={\rm conv}\{O,D,A_2,A_1,B\},$$
where $A_1$ lies on the line segment $AB$ and  $A_2\in {\rm int}
\triangle ABD$, $R$ the origin-symmetric body of revolution
generated by
 $P$. Then
\begin{eqnarray} \mathcal
{P}(R)\geq\min\{\mathcal {P}(R_1),\mathcal {P}(R_2)\}
\end{eqnarray}
and
\begin{eqnarray} \mathcal
{P}(R)\geq\frac{4\pi^2}{3},
\end{eqnarray}
 where  $R_1$ and $R_2$ are origin-symmetric bodies
of revolution generated by 1-unconditional polygons
 $P_1$ and $P_2$
 satisfying
$$P_1\cap \{(x,y): x\leq 0,\;y\geq 0\}={\rm conv}\{O,D,A_2,B\}$$
and
$$P_2\cap \{(x,y): x\leq 0,\;y\geq 0\}={\rm conv}\{O,D,C,B\},$$
respectively,  where $C$ is the point of intersection between two
lines $A_2D$ and $AB$.

\end{lem}
\begin{proof}
In Figure 3.1, let $A_2=(x_0,y_0)$ and $A_1=(-t,1)$, then
$$C=(\frac{x_0-y_0+1}{y_0},1)\;\;\textrm{and}\;\;0\leq t\leq \frac{-x_0+y_0-1}{y_0}.$$

From Remark 2, we can get $P^{\ast}$, which satisfies
$$P^{\ast}\cap \{(x,y): x\leq 0,\;y\geq 0\}=\textrm{conv}\{M,E,D,O,B\},$$
 where $E$ lies on the
line segment $AD$ and $M\in \textrm{int} \triangle ABD$. Let $F$ be
the point of intersection
 between two
lines $EM$ and $AB$. Let $$F_1(t)=\frac{1}{2}V(R),\;\;
F_2(t)=\frac{1}{2}V(R^{\ast})\;\; \textrm{and}
\;\;F(t)=F_1(t)F_2(t).$$

{\it Firstly, we prove {\rm(}3.4{\rm)}.} The proof consists of three
steps for good understanding.

{\bf First step.} {\it We calculate the first and second derivatives
of the functions $F(t)$.}

Since $EF\bot OA_2$ and the distance of the line $EF$ from $O$ is
$1/\|OA_2\|$, we have the equation of the line $EF$
\begin{eqnarray}
y=-\frac{x_0}{y_0}x+\frac{1}{y_0}.
\end{eqnarray}
Similarly, since $BM\bot OA_1$ and the distance of the line $BM$
from $O$ is $1/\|OA_1\|$, we get the equation of the line $BM$
\begin{eqnarray}
y=tx+1.
\end{eqnarray}

Using equations (3.6) and (3.7), we obtain
\begin{eqnarray}
M=(x_M,y_M)=\left(\frac{1-y_0}{ty_0+x_0},
\frac{x_0+t}{ty_0+x_0}\right)
\end{eqnarray}
and
\begin{eqnarray}
E=(x_E,y_E)=\left(-1, \frac{x_0+1}{y_0}\right).
\end{eqnarray}

Noting that
\begin{eqnarray*}
& & P\cap \{(x,y): x\leq 0,\;y\geq 0\}
\nonumber\\
&=&\textrm{conv}\{D,A_2,A_2^{\prime}\}\cup\textrm{conv}\{A_1,A_2,A_2^{\prime},A_1^{\prime}\}\cup\textrm{conv}\{O,B,A_1,A_1^{\prime}\},
\end{eqnarray*}
where $A_1^{\prime}$ and $A_2^{\prime}$ are the orthogonal
projections of points $A_1$ and $A_2$, respectively, on the
$X$-axis, and applying Lemma 2.1, we have
\begin{eqnarray}
F_1(t)&=&\frac{\pi}{3}
y_0^2(x_0+1)+\frac{\pi}{3}(-t-x_0)(y_0^2+y_0+1)+\pi t
\nonumber\\
&=&\frac{\pi}{3}(-y_0^2-y_0+2) t+\frac{\pi}{3}(y_0^2-x_0y_0-x_0).
\end{eqnarray}
Thus, we have
\begin{eqnarray}
F_1^{\prime}(t)=\frac{\pi}{3}(-y_0^2-y_0+2).
\end{eqnarray}

Noting that
\begin{eqnarray*}
& & P^{\ast}\cap \{(x,y): x\leq 0,\;y\geq 0\}
\nonumber\\
&=&\textrm{conv}\{D,E, M,
M^{\prime}\}\cup\textrm{conv}\{M,M^{\prime},O,B\},
\end{eqnarray*}
where $M^{\prime}$ is the orthogonal projection of point $M$ on the
$X$-axis, and applying Lemma 2.1, we obtain

\begin{eqnarray}
F_2(t)&=&\frac{\pi}{3}(x_M-x_E)(y_E^2+y_Ey_M+y_M^2)+\frac{\pi}{3}(-x_M)(y_M^2+y_M+1)
\nonumber\\
&=&\frac{\pi}{3}(\frac{1-y_0}{ty_0+x_0}+1)\left[(\frac{x_0+1}{y_0})^2+(\frac{x_0+1}{y_0})(\frac{x_0+t}{ty_0+x_0})+(\frac{x_0+t}{ty_0+x_0})^2\right]
\nonumber\\
&~&+\frac{\pi}{3}(\frac{y_0-1}{ty_0+x_0})\left[(\frac{x_0+t}{ty_0+x_0})^2+(\frac{x_0+t}{ty_0+x_0})+1\right]
\nonumber\\
&=&\frac{\pi}{3} \frac{\Delta_1 t^3+\Delta_2 t^2+\Delta_3
t+\Delta_4}{y_0^2(ty_0+x_0)^3},
\end{eqnarray}
where
\begin{eqnarray*}
&~&\Delta_1=y_0^3(x_0^2+3x_0+3),
\nonumber\\
&~&\Delta_2=y_0^2(3x_0^3+9x_0^2+9x_0+y_0^3-3y_0+2),
\nonumber\\
&~&\Delta_3=3y_0[x_0^4+3x_0^3+3x_0^2+x_0(y_0^3-y_0^2-y_0+1)],
\nonumber\\
&~&\Delta_4=x_0^2(x_0^3+3x_0^2+3x_0+2y_0^3-3y_0^2+1).
\end{eqnarray*}
Thus, we have
\begin{eqnarray}
F_2^{\prime}(t)&=&\frac{\pi}{3}\frac{(3\Delta_1x_0-\Delta_2y_0)t^2+(2\Delta_2x_0-2\Delta_3y_0)t+(\Delta_3x_0-3\Delta_4y_0)}{y_0^2(ty_0+x_0)^4}
\nonumber\\
&=&\frac{\pi}{3}(y_0-1)^2\frac{-y_0(y_0+2)t^2-2x_0(2y_0+1)t-3x_0^2}{(ty_0+x_0)^4}.
\end{eqnarray}
Then, we have
\begin{eqnarray}
F^{\prime}(t)&=&F_1^{\prime}(t)F_2(t)+F_1(t)F_2^{\prime}(t)
\nonumber\\
&=&\frac{\pi^2}{9}\frac{\Lambda_1t^4+\Lambda_2t^3+\Lambda_3t^2+\Lambda_4t+\Lambda_5}{y_0^2(ty_0+x_0)^4},
\end{eqnarray}
where
\begin{eqnarray*}
&~&\Lambda_1=y_0^4[x_0^2(-y_0^2-y_0+2)+3x_0(-y_0^2-y_0+2)+3(-y_0^2-y_0+2)],
\nonumber\\
&~&\Lambda_2=y_0^3[4x_0^3(-y_0^2-y_0+2)+12x_0^2(-y_0^2-y_0+2)+12x_0(-y_0^2-y_0+2)],
\nonumber\\
&~&\Lambda_3=y_0^2[6x_0^4(-y_0^2-y_0+2)+18x_0^3(-y_0^2-y_0+2)+18x_0^2(-y_0^2-y_0+2)
\nonumber\\
&~&\;\;\;\;\;\;\;\;\;+x_0(y_0^5-2y_0^4+8y_0^2-13y_0+6)+(-y_0^6+3y_0^4-2y_0^3)],
\nonumber\\
&~&\Lambda_4=y_0[4x_0^5(-y_0^2-y_0+2)+12x_0^4(-y_0^2-y_0+2)+12x_0^3(-y_0^2-y_0+2)
\nonumber\\
&~&\;\;\;\;\;\;\;\;\;+x_0^2(2y_0^5-4y_0^4+4y_0^3+4y_0^2-14y_0+8)+x_0(-4y_0^6+6y_0^5-2y_0^3)],
\nonumber\\
&~&\Lambda_5=x_0^6(-y_0^2-y_0+2)+3x_0^5(-y_0^2-y_0+2)+3x_0^4(-y_0^2-y_0+2)
\nonumber\\
&~&\;\;\;\;\;\;\;\;\;+x_0^3(y_0^5-2y_0^4+4y_0^3-4y_0^2-y_0+2)+x_0^2(-3y_0^6+6y_0^5-3y_0^4).
\end{eqnarray*}
Simplifying the above equation, we get
\begin{eqnarray}
F^{\prime}(t)&=&\frac{\pi^2}{9}\frac{-y_0^2-y_0+2}{y_0^2(ty_0+x_0)^3}
\{(x_0^2+3x_0+3)y_0^3t^3+3x_0(x_0^2+3x_0+3)y_0^2t^2
\nonumber\\
&~&+[3x_0^4+9x_0^3+9x_0^2+x_0(-y_0^3+3y_0^2-5y_0+3)+y_0^3(y_0-1)]y_0t
\nonumber\\
&~&+[x_0^5+3x_0^4+3x_0^3+x_0^2\frac{-y_0^4+y_0^3-3y_0^2+y_0+2}{y_0+2}+x_0\frac{3y_0^5-3y_0^4}{y_0+2}]\}.
\nonumber\\
\end{eqnarray}
From (3.14), we can get
\begin{eqnarray}
F^{\prime\prime}(t)&=&\frac{\pi^2}{9}\frac{(4\Lambda_1x_0-\Lambda_2y_0)t^3+(3\Lambda_2x_0-2\Lambda_3y_0)t^2+(2\Lambda_3x_0-3\Lambda_4y_0)t+(\Lambda_4x_0-4\Lambda_5y_0)}{y_0^2(ty_0+x_0)^5}
\nonumber\\
&=&\frac{\pi^2}{9}\frac{\Gamma_1t^2+\Gamma_2t+\Gamma_3}{y_0^2(ty_0+x_0)^5},
\end{eqnarray}
where
\begin{eqnarray*}
&~&\Gamma_1=-2x_0y_0^3(y_0^5-2y_0^4+8y_0^2-13y_0+6)-2y_0^6(-y_0^3+3y_0-2),
\nonumber\\
&~&\Gamma_2=x_0^2y_0^2(-4y_0^5+8y_0^4-12y_0^3+4y_0^2+16y_0-12)+x_0y_0^5(10y_0^3-18y_0^2+6y_0+2),
\nonumber\\
&~&\Gamma_3=x_0^3y_0^2(-2y_0^4+4y_0^3-12y_0^2+20y_0-10)+x_0^2y_0^4(8y_0^3-18y_0^2+12y_0-2).
\end{eqnarray*}
Simplifying the above equation, we get
\begin{eqnarray}
F^{\prime\prime}(t)&=&\frac{\pi^2}{9}\frac{\frac{\Gamma_1}{y_0}t+(\frac{\Gamma_2}{y_0}-\frac{x_0\Gamma_1}{y_0^2})}{y_0^2(ty_0+x_0)^4}
\nonumber\\
&=&\frac{\pi^2}{9}\frac{(y_0-1)^2}{(ty_0+x_0)^4}\{[-2x_0(y_0+2)(y_0^2-2y_0+3)+2y_0^3(y_0+2)]t
\nonumber\\
&~&+[x_0^2(-2y_0^2-10)+x_0y_0^2(8y_0-2)]\}.
\end{eqnarray}

\begin{figure}[htb]
\centering
 \includegraphics[height=10cm]{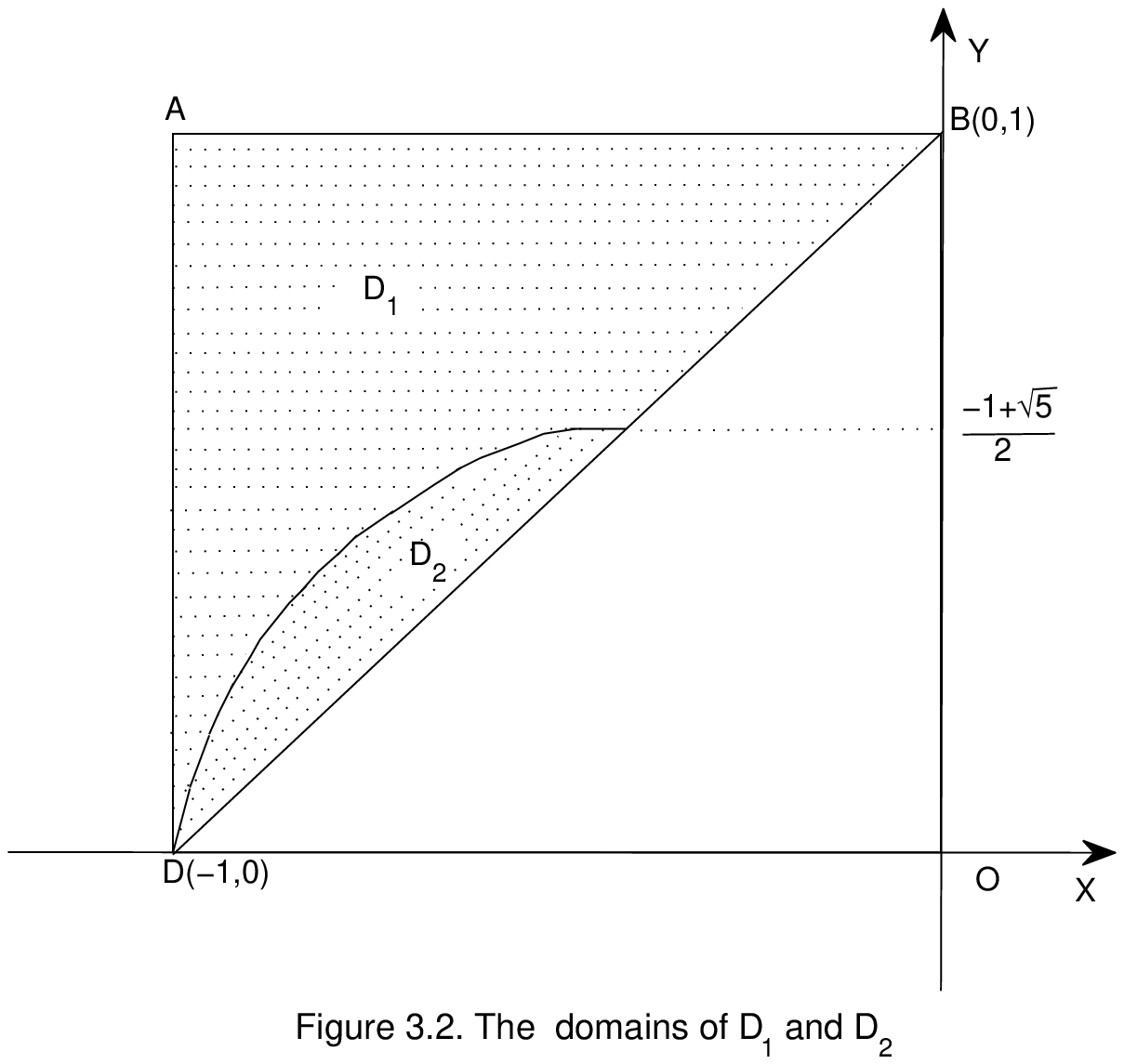 }
\end{figure}
{\bf Second step.} {\it We prove that

$${\rm(}i{\rm)}\;\; F^{\prime}(\frac{-x_0+y_0-1}{y_0})\leq 0
\;\;\textrm{for }(x_0,y_0)\in \mathcal {D}_1$$ and

$${\rm(}ii{\rm)}\;\; F^{\prime\prime}(\frac{-x_0+y_0-1}{y_0})\leq 0
\;\;\textrm{for}\;\; (x_0,y_0)\in \mathcal {D}_2,$$

where
\begin{eqnarray}
\mathcal {D}_1&=&\{(x,y): -1\leq x\leq
y-1,\;\frac{-1+\sqrt{5}}{2}\leq y\leq 1\}
\nonumber\\
&~&\cup\{(x,y): -1\leq x\leq
\frac{y^3+2y^2+3y-6}{(2-y)(y+3)},\;0\leq
y\leq\frac{-1+\sqrt{5}}{2}\} \nonumber\\
\end{eqnarray}
and
\begin{eqnarray}
\mathcal {D}_2&=&\{(x,y): \frac{y^3+2y^2+3y-6}{(2-y)(y+3)}\leq x\leq
y-1,\;0\leq
y\leq\frac{-1+\sqrt{5}}{2}\}.\nonumber\\
\end{eqnarray}}

In fact, from (3.15), we have that
\begin{eqnarray}
F^{\prime}(\frac{-x_0+y_0-1}{y_0}) &=&\frac{\pi^2}{9y_0^2}\;
G(x_0,y_0),
\end{eqnarray}
where
$$G(x_0,y_0)=x_0^2(2-y_0)(y_0+3)-x_0(y_0^3+3y_0^2+4y_0-12)-(y_0+2)(y_0^3+3y_0-3).$$
Noting that $G(x_0,y_0)$ is a quadratic function of the variable
$x_0$ defined on $[-1, y_0-1]$ and $0\leq y_0\leq 1$, the graph of
the quadratic function is a parabola opening upwards.

When $x_0=-1$, we obtain
$$G(-1,y_0)=-y_0^2(y_0^2+y_0+1)<0.$$
When $x_0=y_0-1$, we have
$$G(y_0-1,y_0)=-3y_0^2(y_0^2+y_0-1).$$
Then we have
$$G(y_0-1,y_0)\leq 0\;\; {\rm for} \;\frac{-1+\sqrt{5}}{2}\leq y_0\leq 1$$
and
$$G(y_0-1,y_0)\geq0\;\; {\rm for} \;0\leq y_0<\frac{-1+\sqrt{5}}{2}.$$
When $$x_0=\frac{y_0^3+2y_0^2+3y_0-6}{(2-y_0)(y_0+3)}\in
[-1,y_0-1],$$ we have
$$G(\frac{y_0^3+2y_0^2+3y_0-6}{(2-y_0)(y_0+3)},y_0)=G(-1,y_0)<0.$$
Hence,
$$G(x_0,y_0)\leq 0,\;\;\textrm{for}\;\;(x_0,y_0)\in\mathcal {D}_1.$$
From (3.20), we have
\begin{eqnarray}
F^{\prime}(\frac{-x_0+y_0-1}{y_0})\leq 0\;\;{\rm
for}\;\;(x_0,y_0)\in \mathcal {D}_1.
\end{eqnarray}

By (3.17), we get
\begin{eqnarray}
F^{\prime\prime}(\frac{-x_0+y_0-1}{y_0})
&=&\frac{\pi^2}{9}\frac{1}{y_0(1-y_0)} H(x_0,y_0),
\end{eqnarray}
where

\begin{eqnarray}
H(x_0,y_0)=12x_0^2-x_0(4y_0^3+2y_0-12)-2y_0^3(y_0+2).
\end{eqnarray}
Noting that $H(x_0,y_0)$ is a quadratic function of the variable
$x_0$ defined on $[-1,y_0-1]$ and the coefficient of the quadratic
term is positive, the graph of the quadratic function is a parabola
opening upwards.

Let $x_0=y_0-1$, we have
\begin{eqnarray}
H(y_0-1,y_0)=-6y_0^4-10y_0(1-y_0)\leq 0.
\end{eqnarray}

Let
$$x_0=\frac{y_0^3+2y_0^2+3y_0-6}{(2-y_0)(y_0+3)},$$ we have
\begin{eqnarray}
&&H(\frac{y_0^3+2y_0^2+3y_0-6}{(2-y_0)(y_0+3)},y_0)
\nonumber\\
&=&\frac{2y_0^8+4y_0^7+24y_0^6+50y_0^5-38y_0^4-18y_0^3-48y_0^2-72y_0}{(2-y_0)^2(y_0+3)^2}
\nonumber\\
&\leq& 0.
\end{eqnarray}
From (3.24) and (3.25), we have
$$H(x_0,y_0)\leq 0\;\;\textrm{for}\;\;(x_0,y_0)\in \mathcal {D}_2.$$
Therefore, from (3.22) and $0<y_0<1$, we have
\begin{eqnarray}
F^{\prime\prime}(\frac{-x_0+y_0-1}{y_0})\leq 0\;\;{\rm
for}\;\;(x_0,y_0)\in \mathcal {D}_2.
\end{eqnarray}
\\

{\bf Third step.} {\it We prove $\mathcal {P}(R)\geq\min\{\mathcal
{P}(R_1),\mathcal {P}(R_2)\}$.}

By (3.17), we have
\begin{eqnarray}
F^{\prime\prime}(t)&=&\frac{\pi^2}{9}\frac{(y_0-1)^2}{(ty_0+x_0)^4}
I(t),
\end{eqnarray}
where
\begin{eqnarray}
I(t)&=&[-2x_0(y_0+2)(y_0^2-2y_0+3)+2y_0^3(y_0+2)] t
 \nonumber\\
&~&+[x_0^2(-2y_0^2-10)+x_0y_0^2(8y_0-2)]
\end{eqnarray}
and
\begin{eqnarray}
0\leq t\leq \frac{-x_0+y_0-1}{y_0}.
\end{eqnarray}

Since
$$-2x_0(y_0+2)(y_0^2-2y_0+3)+2y_0^3(y_0+2)>0,$$  $I(t)$ is an increasing function of the variable $t$.

By (3.26), for any
$$(x_0,y_0)\in\mathcal {D}_2,$$
we have
$$F^{\prime\prime}(\frac{-x_0+y_0-1}{y_0})\leq 0.$$
From (3.27), we have
$$I(\frac{-x_0+y_0-1}{y_0})\leq 0,$$
which implies that $I(t)\leq 0$ for any
$$0\leq t\leq \frac{-x_0+y_0-1}{y_0}.$$
Therefore $F^{\prime\prime}(t)\leq 0$ for any

$$0\leq t\leq \frac{-x_0+y_0-1}{y_0}.$$

It follows that the function $F(t)$ is concave on the interval
$$[0,\frac{-x_0+y_0-1}{y_0}],$$ which implies
$$F(t)\geq \min\{F(0),F(\frac{-x_0+y_0-1}{y_0})\}.$$
Therefore, we have
$$\mathcal
{P}(R)\geq\min\{\mathcal {P}(R_1),\mathcal {P}(R_2)\}.$$

By (3.21), for any $(x_0,y_0)\in\mathcal {D}_1$, we have
$$F^{\prime}(\frac{-x_0+y_0-1}{y_0})\leq 0.$$
Now we prove that the inequality (3.4) holds in each of the
following situations:

$$\textrm{(i)}\;\;I(\frac{-x_0+y_0-1}{y_0})\leq 0;\;\;\;\;\;\;\;\;\;\;\;\;\;\;\;\;\;\;\;\;\;\;$$
$$\textrm{(ii)}\;\;I(\frac{-x_0+y_0-1}{y_0})> 0\;\;\textrm{and}\;\;I(0)< 0;$$
$$\textrm{(iii)}\;\;I(0)\geq 0.\;\;\;\;\;\;\;\;\;\;\;\;\;\;\;\;\;\;\;\;\;\;\;\;\;\;\;\;\;\;\;\;\;\;\;\;\;\;\;\;\;$$
We have proved (3.4) in the case (i), and now we prove (3.4) in
cases (ii) and (iii).

For the case (ii), since $I(t)$ is increasing and by (3.27), there
exists a real number
$$t_0\in (0,\frac{-x_0+y_0-1}{y_0})$$ satisfying
$$F^{\prime\prime}(t)\leq 0\;\;\textrm{for}\;\;t\in [0,t_0]$$
and
$$F^{\prime\prime}(t)> 0\;\;\textrm{for}\;\;t\in (t_0,\frac{-x_0+y_0-1}{y_0}].$$
It follows that $F^{\prime}(t)$ is decreasing on the interval
$[0,t_0]$ and increasing on the interval
$$(t_0,\frac{-x_0+y_0-1}{y_0}].$$
If $F^{\prime}(0)\leq 0$, and since
$$F^{\prime}(\frac{-x_0+y_0-1}{y_0})\leq 0,$$ we have
$$F^{\prime}(t)\leq 0\;\;\textrm{for}\;\;\textrm{any}\;\;t\in [0,\frac{-x_0+y_0-1}{y_0}],$$
which implies that the function $F(t)$ is  decreasing and
$$F(t)\geq F(\frac{-x_0+y_0-1}{y_0})\;\;\textrm{for}\;\;\textrm{any}\;\;t\in[0,\frac{-x_0+y_0-1}{y_0}].$$
Therefore we have $$\mathcal {P}(R)\geq\min\{\mathcal
{P}(R_1),\mathcal {P}(R_2)\}=\mathcal {P}(R_2).$$

If $F^{\prime}(0)>0$,  there exists a real number
$$t_1\in (0,\frac{-x_0+y_0-1}{y_0})$$ satisfying
$$F^{\prime}(t)>0\;\;\textrm{for}\;\;\textrm{any}\;\;t\in [0,t_1)$$
and
$$F^{\prime}(t)\leq 0\;\;\textrm{for}\;\;\textrm{any}\;\;t\in [t_1,\frac{-x_0+y_0-1}{y_0}],$$
which implies that the function $F(t)$ is increasing on the interval
$[0,t_1)$ and decreasing on the interval
$$[t_1,\frac{-x_0+y_0-1}{y_0}].$$  It follows that
$$F(t)\geq \min\{F(0),F(\frac{-x_0+y_0-1}{y_0})\}\;\;\textrm{for}\;\;\textrm{any}\;\;t\in [0,\frac{-x_0+y_0-1}{y_0}].$$
We then have
$$\mathcal {P}(R)\geq\min\{\mathcal {P}(R_1),\mathcal
{P}(R_2)\}.$$

For the case (iii), since the function $I(t)$ is increasing, we have
$$I(t)\geq 0\;\;\textrm{for}\;\;\textrm{any}\;\;t\in
[0,\frac{-x_0+y_0-1}{y_0}].$$ Hence, from (3.27), we have
$$F^{\prime\prime}(t)\geq 0\;\;\textrm{for}\;\;\textrm{any}\;\;t\in
[0,\frac{-x_0+y_0-1}{y_0}].$$ Therefore, the function
$F^{\prime}(t)$ is increasing on the interval
$$[0,\frac{-x_0+y_0-1}{y_0}],$$and since
$$F^{\prime}(\frac{-x_0+y_0-1}{y_0})\leq 0,$$ we have
$$F^{\prime}(t)\leq 0\;\;\textrm{for}\;\;\textrm{any}\;\;t\in
[0,\frac{-x_0+y_0-1}{y_0}],$$which implies that the function $F(t)$
is decreasing on the interval
$$[0,\frac{-x_0+y_0-1}{y_0}].$$
Therefore, we have
$$F(t)\geq F(\frac{-x_0+y_0-1}{y_0})\;\;\textrm{for}\;\;\textrm{any}\;\;t\in
[0,\frac{-x_0+y_0-1}{y_0}],$$ which implies that
$$\mathcal {P}(R)\geq\min\{\mathcal {P}(R_1),\mathcal
{P}(R_2)\}=\mathcal {P}(R_2).$$
\\

\textit{Secondly, we prove (3.5).}
\begin{figure}[htb]
\centering
 \includegraphics[height=10cm]{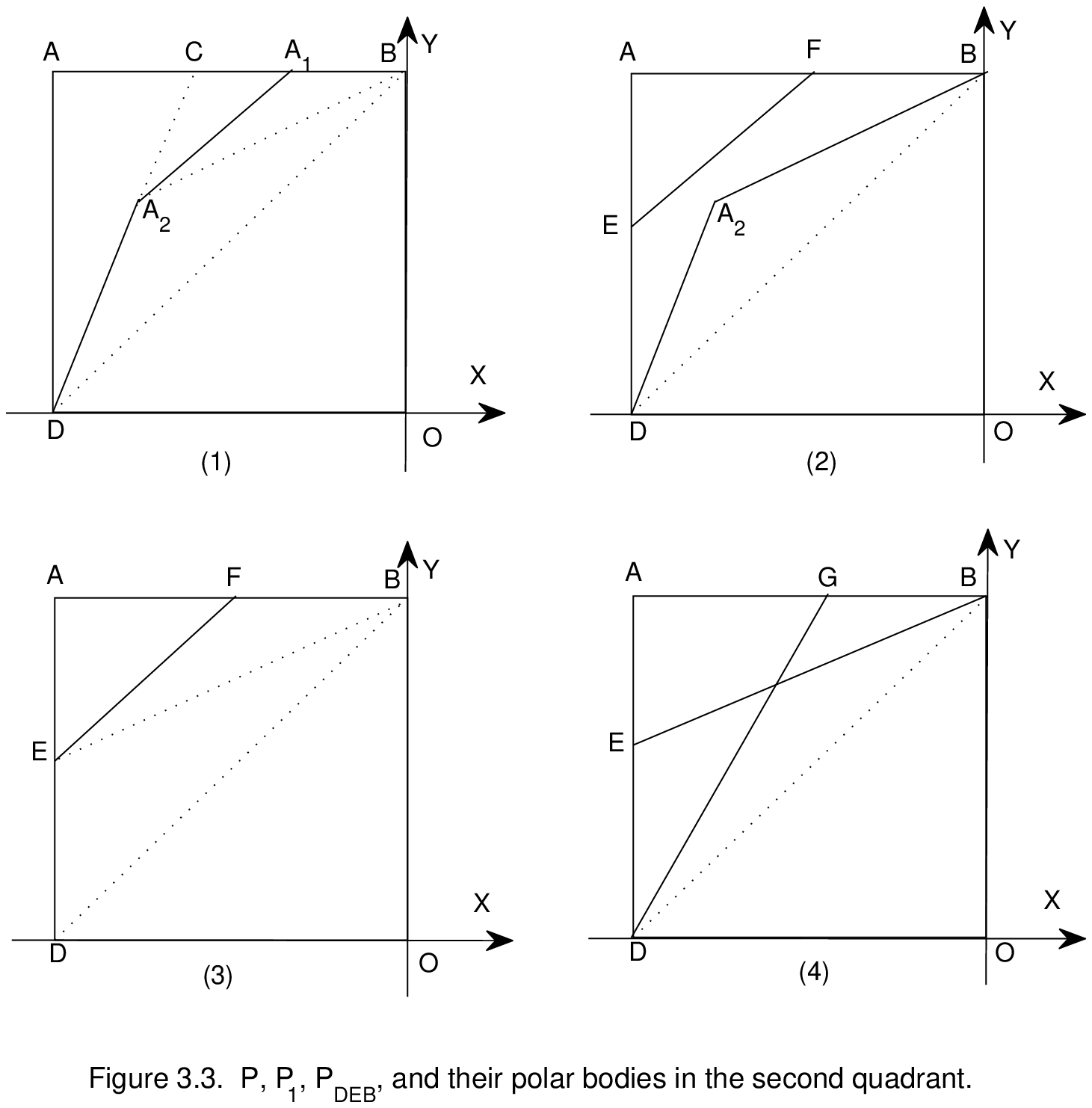 }
\end{figure}

In (3.4), if
$$\min\{\mathcal {P}(R_1),\mathcal {P}(R_2)\}=\mathcal {P}(R_2).$$
 Let
$$T=\{(x,y):|x|+|y|\leq 1\}$$ and
$$Q=\{(x,y):\max\{|x|,|y|\}\leq 1\}.$$
Let $R_T$ and $R_Q$ be the origin-symmetric bodies of revolution
generated by $T$ and $Q$, respectively. In (3.4), replacing $R$,
$R_1$, and $R_2$, by $R_2$, $R_T$, and $R_Q$, respectively (see (1)
of Figure 3.3), we obtain
\begin{eqnarray}
\mathcal {P}(R_2)\geq \min\{\mathcal {P}(R_T),\mathcal
{P}(R_Q)\}=\frac{4\pi^2}{3}.
\end{eqnarray}
It follows that
$$\mathcal {P}(R)\geq \mathcal {P}(R_2)\geq \frac{4\pi^2}{3}.$$

In (3.4), if
$$\min\{\mathcal {P}(R_1),\mathcal {P}(R_2)\}=\mathcal {P}(R_1),$$
let $E$, $F$ be the vertices of $P_1^{\ast}$ in the second quadrant,
where $E$, $F$ lie on line segments $AD$ and $AB$, respectively (see
(2) of Figure 3.3). Let $P_{DEB}$ be a 1-unconditional polygon
satisfying
$$P_{DEB}\cap\{(x,y): x\leq0,y\geq0\}=\textrm{conv}\{E,D,O,B\},$$
and let $R_{DEB}$ be an origin-symmetric body of revolution
generated by $P_{DEB}$. In (3.4), replacing $R$, $R_1$, and $R_2$,
by ${R_1}^{\ast}$, $R_{DEB}$, and $R_Q$, respectively  (see (3) of
Figure 3.3), we have
\begin{eqnarray}
\mathcal {P}(R)\geq \mathcal {P}(R_1)=\mathcal {P}({R_1}^{\ast})\geq
\min\{\mathcal {P}(R_{DEB}),\mathcal {P}(R_Q)\}.
\end{eqnarray}
In (3.31), if
$$\min\{\mathcal {P}(R_{DEB}),\mathcal
{P}(R_Q)\}=\mathcal {P}(R_Q),$$ we have proved (3.5); if
$$\min\{\mathcal {P}(R_{DEB}),\mathcal
{P}(R_Q)\}=\mathcal {P}(R_{DEB}),$$ let
$${P_{DEB}}^{\ast}\cap\{(x,y): x\leq0,\;y\geq0\}=\textrm{conv}\{G,D,O,B\},$$
where $G$ lies on the line segment $AB$, which is a vertex of
${P_{DEB}}^{\ast}$ (see (4) of Figure 3.3). In (3.4), replacing $R$,
$R_1$, and $R_2$, by ${R_{DEB}}^{\ast}$, $R_T$, and $R_Q$,
respectively, we obtain
\begin{eqnarray}
\mathcal {P}({R_{DEB}}^{\ast})\geq \min\{\mathcal {P}(R_T),\mathcal
{P}(R_Q)\}=\frac{4\pi^2}{3}.
\end{eqnarray}

Hence, we have  $$ \mathcal {P}(R)\geq \mathcal {P}(R_1)\geq\mathcal
{P}(R_{DEB})\geq\frac{4\pi^2}{3}.$$
\end{proof}
\begin{figure}[htb]
\centering
 \includegraphics[height=10cm]{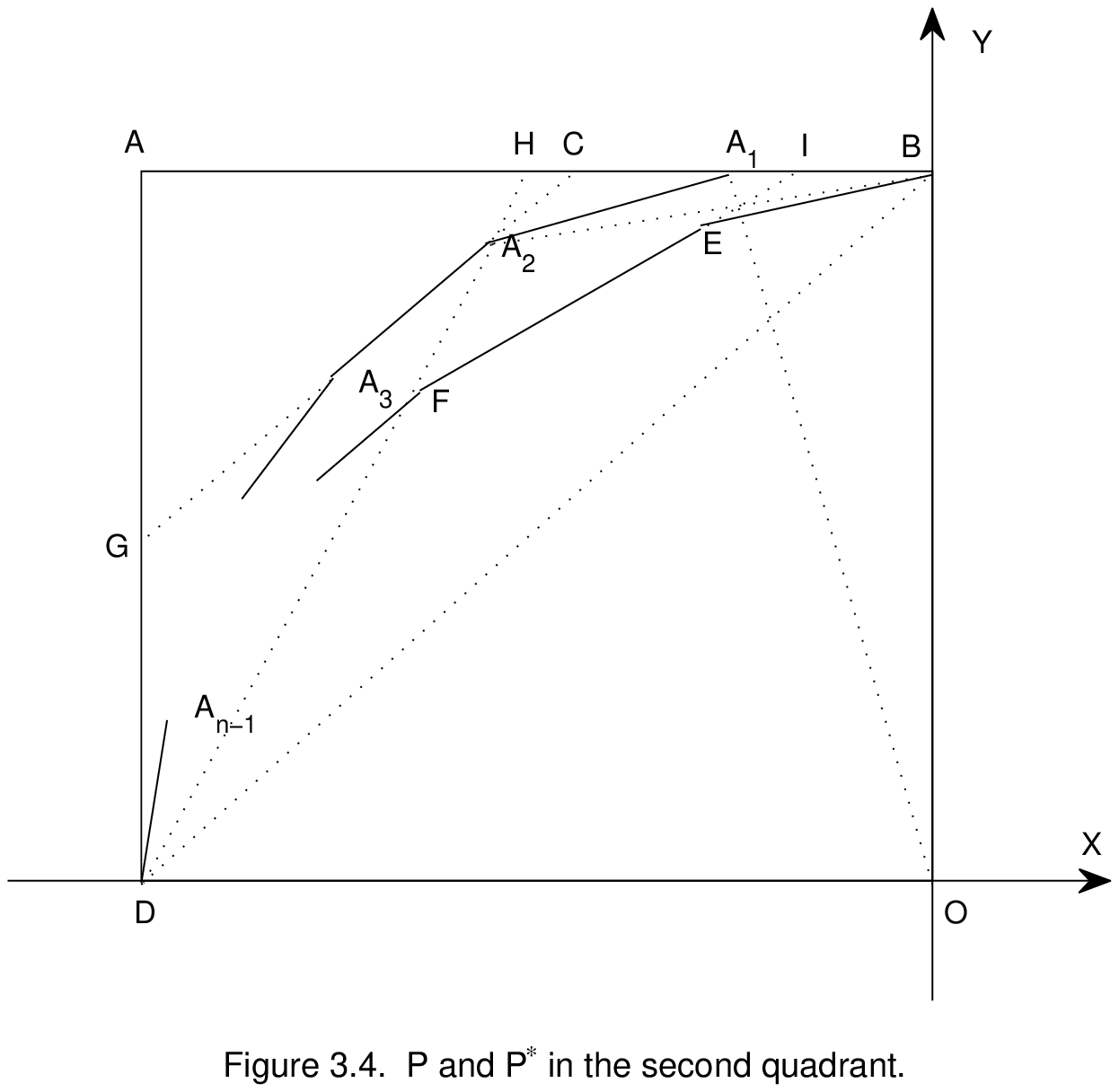 }
\end{figure}
\begin{lem} Let $P$ be a 1-unconditional polygon in the coordinate plane $XOY$ satisfying
$$P\cap\{(x,y): x\leq0,\;y\geq0\}={\rm conv}\{A_1, A_2, \cdots, A_{n-1}, D, O, B\},$$
where $A_1$ lies on the line segment $AB$, $A_2,\cdots, A_{n-1}\in
{\rm int} \triangle ABD$, and the slopes of lines $OA_i$ {\rm
(}$i=1,\cdots,n-1${\rm)} are increasing on $i$, $R$ the
origin-symmetric body of revolution generated by $P$. Then
\begin{eqnarray} \mathcal
{P}(R)\geq\min\{\mathcal {P}(R_1),\mathcal {P}(R_2)\},
\end{eqnarray}
where $R_1$ and $R_2$ are origin-symmetric bodies of revolution
generated by 1-unconditional polygons $P_1$ and $P_2$ satisfying
$$P_1\cap\{(x,y): x\leq0,\;y\geq0\}={\rm conv}\{A_2, A_3,\cdots, A_{n-1}, D, O, B\}$$
and
$$P_2\cap\{(x,y): x\leq0,\;y\geq0\}={\rm conv}\{C, A_3,\cdots, A_{n-1}, D, O, B\},$$
respectively, where $C$ is the point of intersection between two
lines $A_2A_3$ and $AB$.
\end{lem}

\begin{proof}In Figure 3.4, let $A_1=(-t,1)$ and $A_2=(x_0,y_0)$. Let the slope of the line
$A_3A_2$ be $k$, then
\begin{eqnarray}
\frac{1-y_0}{-x_0}<k<\frac{y_0}{x_0+1}
\end{eqnarray}
and the equation of the line $A_3A_2$ is
\begin{eqnarray}
y-y_0=k(x-x_0).
\end{eqnarray}
In (3.35), let $y=1$, we get the abscissa of $C$
$$x_C=x_0+\frac{1-y_0}{k}.$$

Let $E,\;F$ and $B$ be the vertices of $P^{\ast}$ satisfying $BE\bot
OA_1$ and $EF\bot OA_2$. Let $I$ be the point of intersection
between two lines $EF$ and $AB$. We have
$$BE:~y=tx+1$$
and
$$EF:~y=-\frac{x_0}{y_0}x+\frac{1}{y_0}.$$

Then, we get
$$I=(\frac{1-y_0}{x_0},1)$$
and
\begin{eqnarray}
E=(\frac{1-y_0}{t y_0+x_0},\frac{t+x_0}{t y_0+x_0}).
\end{eqnarray}
Let
\begin{eqnarray}
F(t)=\frac{1}{2}V(R)\frac{1}{2}V(R^{\ast})=\frac{1}{4}\mathcal
{P}(R),
\end{eqnarray}
which is a function of the variable $t$, where
$$0\leq t\leq -x_C=\frac{-x_0k+y_0-1}{k}.$$

Our proof has three steps.
\\

{\bf First step.} \textit{Calculate $F^{\prime}(t)$ and
$F^{\prime\prime}(t)$.}

 Let
$V=\frac{1}{2}V(R_1)$ and $V^{0}=\frac{1}{2}V({R_1}^{\ast})$, then
we obtain
\begin{eqnarray}
F(t)&=&\left(V+\frac{\pi}{3}(2-y_0-y_0^2)t\right)
\nonumber\\
&~&\times\left(V^{0}-\frac{\pi}{3}
\frac{y_0-1}{x_0}\left(2-\frac{t+x_0}{t
y_0+x_0}-\left(\frac{t+x_0}{t y_0+x_0}\right)^2\right)\right). \nonumber\\
\end{eqnarray}
Therefore, we have
\begin{eqnarray}
F^{\prime}(t)&=&\frac{\pi}{3}\frac{(2-y_0-y_0^2)(\Phi_1t^3+\Phi_2t^2+\Phi_3t+\Phi_4)}{(y_0t+x_0)^3},
\end{eqnarray}
where
\begin{eqnarray}
&~&\Phi_1=y_0[-\frac{\pi}{3}\frac{(1-y_0)^2(2y_0+1)}{x_0}+V^0y_0^2],
\nonumber\\
&~&\Phi_2=-\pi(1-y_0)^2(2y_0+1)+3V^0x_0y_0^2,
\nonumber\\
&~&\Phi_3=-2\pi(1-y_0)^2x_0+3V^0x_0^2y_0+(y_0-1)V,
\nonumber\\
&~&\Phi_4=V^0x_0^3-\frac{3x_0(1-y_0)V}{y_0+2}.
\end{eqnarray}
Thus, we have
\begin{eqnarray}
F^{\prime\prime}(t)&=&\frac{2\pi}{3}\frac{(1-y_0)^2}{(t y_0+x_0)^4}
J(t),
\end{eqnarray}
where
\begin{eqnarray}
J(t)&=&(y_0+2)[Vy_0+\pi x_0(y_0-1)] t
\nonumber\\
&&+x_0[V(4y_0-1)+\pi x_0(y_0^2+y_0-2)].
\end{eqnarray}
\\

{\bf Second step.} {\it We prove that
$${\rm(}i{\rm)}\;\; F^{\prime}(\frac{-x_0k+y_0-1}{k})\leq 0
\;\;\textrm{or}\;\;F^{\prime\prime}(\frac{-x_0k+y_0-1}{k})\leq
0\;\;\textrm{for }(x_0,y_0)\in \mathcal {D}_1$$ and
$${\rm(}ii{\rm)}\;\; F^{\prime\prime}(\frac{-x_0k+y_0-1}{k})\leq 0
\;\;\textrm{for}\;\; (x_0,y_0)\in \mathcal {D}_2,$$

where $\mathcal {D}_1$ and $\mathcal {D}_2$ have been given in
{\rm(}3.18{\rm)} and {\rm(}3.19{\rm)}.}
\\

By (3.39) and (3.40), let $$t_0=\frac{-x_0k+y_0-1}{k},$$ we have
\begin{eqnarray}
F^{\prime}(t_0)=\frac{\pi}{3}(\Upsilon_1V^0+\Upsilon_2V+\Upsilon_3),
\end{eqnarray}
where
\begin{eqnarray}
&~&\Upsilon_1=(1-y_0)(y_0+2),
\nonumber\\
&~&\Upsilon_2=\frac{k^2(-x_0k+y_0+2)}{(x_0k-y_0)^3},
\nonumber\\
&~&\Upsilon_3=-\frac{\pi}{3}\frac{y_0+2}{x_0(x_0k-y_0)^3}[k^3x_0^3(y_0-1)(-2y_0+3)+3k^2x_0^2y_0(y_0-1)(2y_0-3)
\nonumber\\
&~&\;\;\;\;\;\;\;\;\;\;+3kx_0(1-y_0)^3(2y_0+1)+y_0(2y_0+1)(y_0-1)^3].
\end{eqnarray}
Since $k>0$, $x_0<0$ and $0<y_0<1$, we have that $\Upsilon_1\geq 0$
and $\Upsilon_2\leq 0$, thus, as $V$ increases and  $V^0$ decreases,
$F^{\prime}(t_0)$ decreases.

Let $P_0$ be a 1-unconditional polygon satisfying
$$P_0\cap\{(x,y): x\leq0,\;y\geq0\}=\textrm{conv}\{A_2, D, O, B\}$$
and $R_0$ be an origin-symmetric body of revolution generated by
$P_0$. Let $V_0=\frac{1}{2}V(R_0)$ and
${V_0}^{\ast}=\frac{1}{2}V(R_0^{\ast})$. In (3.38), let $V=V_0$ and
$V^{0}={V_0}^{\ast}$, we get a function $F_0(t)$, which is the same
function as $F(t)$ in Lemma 3.4.

Since $V\geq V_0$ and $V^0\leq V_0^{\ast}$, we have

\begin{eqnarray}
F^{\prime}(\frac{-x_0k+y_0-1}{k})\leq
F_0^{\prime}(\frac{-x_0k+y_0-1}{k}).
\end{eqnarray}

Since
$$\frac{1-y_0}{-x_0}\leq k\leq \frac{y_0}{x_0+1},$$
we have
$$0\leq \frac{-x_0k+y_0-1}{k}\leq \frac{-x_0+y_0-1}{y_0}.$$
In (3.45), let
$$k=\frac{y_0}{x_0+1},$$
we have
\begin{eqnarray}
F^{\prime}(\frac{-x_0+y_0-1}{y_0})\leq
F_0^{\prime}(\frac{-x_0+y_0-1}{y_0}).
\end{eqnarray} From Lemma 3.4, we
have
\begin{eqnarray}
F_0^{\prime}(\frac{-x_0+y_0-1}{y_0})\leq
0\;\;\textrm{for}\;\;\textrm{any}\;\;(x_0,y_0)\in\mathcal {D}_1,
\end{eqnarray} hence
\begin{eqnarray}
F^{\prime}(\frac{-x_0+y_0-1}{y_0})\leq
0\;\;\textrm{for}\;\;\textrm{any}\;\;(x_0,y_0)\in\mathcal {D}_1.
\end{eqnarray}
If $F^{\prime\prime}(t_0)> 0$, by (3.41), $J(t_0)>0$, since $x_0<0$
and $0\leq y_0\leq 1$, $J(t)$ is an increasing linear function, thus
$J(t)>0$ for $t\geq t_0$, which implies $F^{\prime\prime}(t)> 0$ for
$t\geq t_0$. Thus $F^{\prime}(t)$ is increasing for $t\geq t_0$.
Since $$F^{\prime}(\frac{-x_0+y_0-1}{y_0})\leq 0,$$ we have
$F^{\prime}(t_0)\leq 0$. Therefore we have proved (i).

Next we prove (ii).

Let $G$ be the point of intersection between two lines $AD$ and
$A_2A_3$, then $G=(-1, y_0-k(x_0+1))$.
 Let $P_M$ be a 1-unconditional polygon satisfying
 $$P_M\cap\{(x,y): x\leq0,\;y\geq0\}=\textrm{conv}\{A_2, G, D, O, B\}$$
 and $R_M$ an origin-symmetric body of revolution generated by $P_M$. From Lemma 2.1, we
 have that
\begin{eqnarray}
\frac{1}{2}V(R_M)&=&\frac{\pi}{3}(x_0+1)[(y_0-k(x_0+1))^2+(y_0-k(x_0+1))y_0+y_0^2]
\nonumber\\
&~& +\frac{\pi}{3}(-x_0)(y_0^2+y_0+1).
\end{eqnarray}

 In (3.42), let
$$V=\frac{1}{2}V(R_M)$$
and
$$t=\frac{-x_0k+y_0-1}{k},$$
we get a function of the variable $k$
\begin{eqnarray}
L(k)&=&\frac{\Theta_1k^3+\Theta_2k^2+\Theta_3k+\Theta_4}{k},
\end{eqnarray}
where
\begin{eqnarray}
&~&\Theta_1=-\frac{\pi}{3}x_0(x_0+1)^3(y_0-1)^2,
\nonumber\\
&~&\Theta_2=\frac{\pi}{3}(x_0+1)^2y_0(y_0-1)(4x_0y_0-x_0+y_0+2),
\nonumber\\
&~&\Theta_3=\frac{\pi}{3}(y_0-1)(-5x_0^2y_0^3-9x_0y_0^3-3x_0^2y_0^2-9x_0y_0^2-x_0^2-3y_0^3-6y_0^2),
\nonumber\\
&~&\Theta_4=\frac{\pi}{3}(y_0-1)(y_0+2)(2x_0y_0^3+3y_0^3-x_0y_0^2+2x_0y_0-3x_0).
\end{eqnarray}
Let
\begin{eqnarray}
L_1(k)&=&\Theta_1k^3+\Theta_2k^2+\Theta_3k+\Theta_4.
\end{eqnarray}

Since $k>0$, to prove $L(k)\leq 0$, it suffices to prove $L_1(k)\leq
0$. In the following, we prove $L_1(k)\leq 0$ for
$$\frac{1-y_0}{-x_0}\leq k\leq \frac{y_0}{x_0+1}.$$
By (3.52), we have
\begin{eqnarray}
L_1^{\prime\prime}(k)&=&6\Theta_1k+2\Theta_2.
\end{eqnarray}
Since
$$L_1^{\prime\prime}(\frac{y_0}{x_0+1})=\frac{2\pi}{3}(x_0+1)^3y_0(y_0-1)(y_0+2)\leq 0$$
and
$$\Theta_1=-\frac{\pi}{3}x_0(x_0+1)^3(y_0-1)^2>0,$$
then $$L_1^{\prime\prime}(k)\leq
0\;\;\textrm{for}\;\;\textrm{any}\;\;\frac{1-y_0}{-x_0}\leq k\leq
\frac{y_0}{x_0+1}.$$

Hence, the function $L_1^{\prime}(k)$ is decreasing on the interval
$$[\frac{1-y_0}{-x_0}, \frac{y_0}{x_0+1}].$$

By (3.52), we have
\begin{eqnarray}
L_1^{\prime}(k)&=&3\Theta_1k^2+2\Theta_2k+\Theta_3.
\end{eqnarray}
From (3.54), we have that
\begin{eqnarray}
L_1^{\prime}(\frac{y_0}{x_0+1})&=&\frac{\pi}{3}(1-y_0)[x_0^2(2y_0^2+1)+x_0(2y_0^3+4y_0^2)+y_0^3+2y_0^2]
\nonumber\\
&=&\frac{\pi}{3}(1-y_0)\left[(2y_0^2+1)\left(x_0+\frac{y_0^3+2y_0^2}{2y_0^2+1}\right)^2+\frac{y_0^2(y_0+2)(1-y_0^3)}{2y_0^2+1}\right]
\nonumber\\
&\geq& 0.
\end{eqnarray}
Therefore $$L_1^{\prime}(k)\geq 0\;\; \textrm{for}\;\;
\textrm{any}\;\;\frac{1-y_0}{-x_0}\leq k\leq \frac{y_0}{x_0+1}.$$ It
follows that the function $L_1(k)$ is increasing on the interval
$$[\frac{1-y_0}{-x_0}, \frac{y_0}{x_0+1}].$$

When $$k=\frac{y_0}{x_0+1},$$ we have $R_M=R_0$ and
$$\frac{-x_0k+y_0-1}{k}=\frac{-x_0+y_0-1}{y_0}.$$

In Lemma 3.4, for $R=R_0$, we had proved
$$F^{\prime\prime}(\frac{-x_0+y_0-1}{y_0})\leq 0\;\;\textrm{for}\;\;(x_0,y_0)\in\mathcal {D}_2.$$
Hence,
$$L_1(\frac{y_0}{x_0+1})\leq 0\;\;\textrm{for}\;\;(x_0,y_0)\in\mathcal
{D}_2,$$ which implies that $L_1(k)\leq 0$ for any
 $$\frac{1-y_0}{-x_0}\leq k\leq \frac{y_0}{x_0+1}\;\;\textrm{when}\; \;
(x_0,y_0)\in\mathcal {D}_2.$$ It follows that, for $R=R_M$,
$$F^{\prime\prime}(\frac{-x_0k+y_0-1}{k})\leq
0\;\;\textrm{for}\;\;\textrm{any}\;\;\frac{1-y_0}{-x_0}\leq k\leq
\frac{y_0}{x_0+1}$$ when $(x_0,y_0)\in\mathcal {D}_2.$

In Lemma 3.4, for $R=R_0$, we know that
$$F^{\prime\prime}(\frac{-x_0+y_0-1}{y_0})\leq 0\;\;\textrm{for}\;\;(x_0,y_0)\in\mathcal
{D}_2,$$ from (3.41), which implies that
$$J(\frac{-x_0+y_0-1}{y_0})\leq
0\;\;\textrm{for}\;\;(x_0,y_0)\in\mathcal {D}_2.$$ Since $J(t)$ is
an increasing linear function and $$\frac{-x_0k+y_0-1}{k}\leq
\frac{-x_0+y_0-1}{y_0}\;\;\textrm{ for}\;\;k<\frac{y_0}{x_0+1},$$ we
have

$$J(\frac{-x_0k+y_0-1}{k})\leq 0\;\;\textrm{for}\;\;(x_0,y_0)\in\mathcal
{D}_2,$$ which implies, for $R=R_0$, that
$$F^{\prime\prime}(\frac{-x_0k+y_0-1}{k})\leq
0$$ for any $$\frac{1-y_0}{-x_0}\leq k\leq \frac{y_0}{x_0+1}\;\;
\textrm{and}\;\; (x_0,y_0)\in\mathcal {D}_2.$$

Therefore, for
$$V=V(R_0)\;\;\textrm{or}\;\;V=V(R_M),$$
we have
$$J(\frac{-x_0k+y_0-1}{k})\leq 0\;\;\textrm{for}\;\;(x_0,y_0)\in\mathcal {D}_2.$$

Since
\begin{eqnarray}
J(t)=[(y_0+2)y_0t+x_0(4y_0-1)]V+[\pi x_0(y_0-1)(y_0+2)t-\pi
x_0^2(2-y_0-y_0^2)],
\nonumber\\
\end{eqnarray}
which can be considered as a linear function of the variable $V$,
and
$$V(R_0)<V(R)<V(R_M),$$
we have, for any $V=V(R)$, that
\begin{eqnarray}
J(\frac{-x_0k+y_0-1}{k})\leq
0\;\;\textrm{for}\;\;(x_0,y_0)\in\mathcal {D}_2.
\end{eqnarray}
It follows that
\begin{eqnarray}
F^{\prime\prime}(\frac{-x_0k+y_0-1}{k})\leq
0\;\;\textrm{for}\;\;(x_0,y_0)\in\mathcal {D}_2.
\end{eqnarray}

\textbf{Third step.} \textit{We prove
$$\mathcal
{P}(R)\geq\min\{\mathcal {P}(R_1),\mathcal {P}(R_2)\}.$$ } We omit
the proof of this step which is similar to the proof of third step
in Lemma 3.4.

\end{proof}
\begin{lem}
For any a 1-unconditional polygon $P\subset [-1,1]^2$ in the
coordinate plane $XOY$ satisfying $B, D\in P$, let $R$ be an
origin-symmetric body of revolution generated by  $P$. Then
\begin{eqnarray}\mathcal {P}(R)\geq \frac{4\pi^2}{3},\end{eqnarray}
with equality if and only if $R$ is a cylinder or bicone.
\end{lem}
\begin{figure}[htb]
\centering
 \includegraphics[height=10cm]{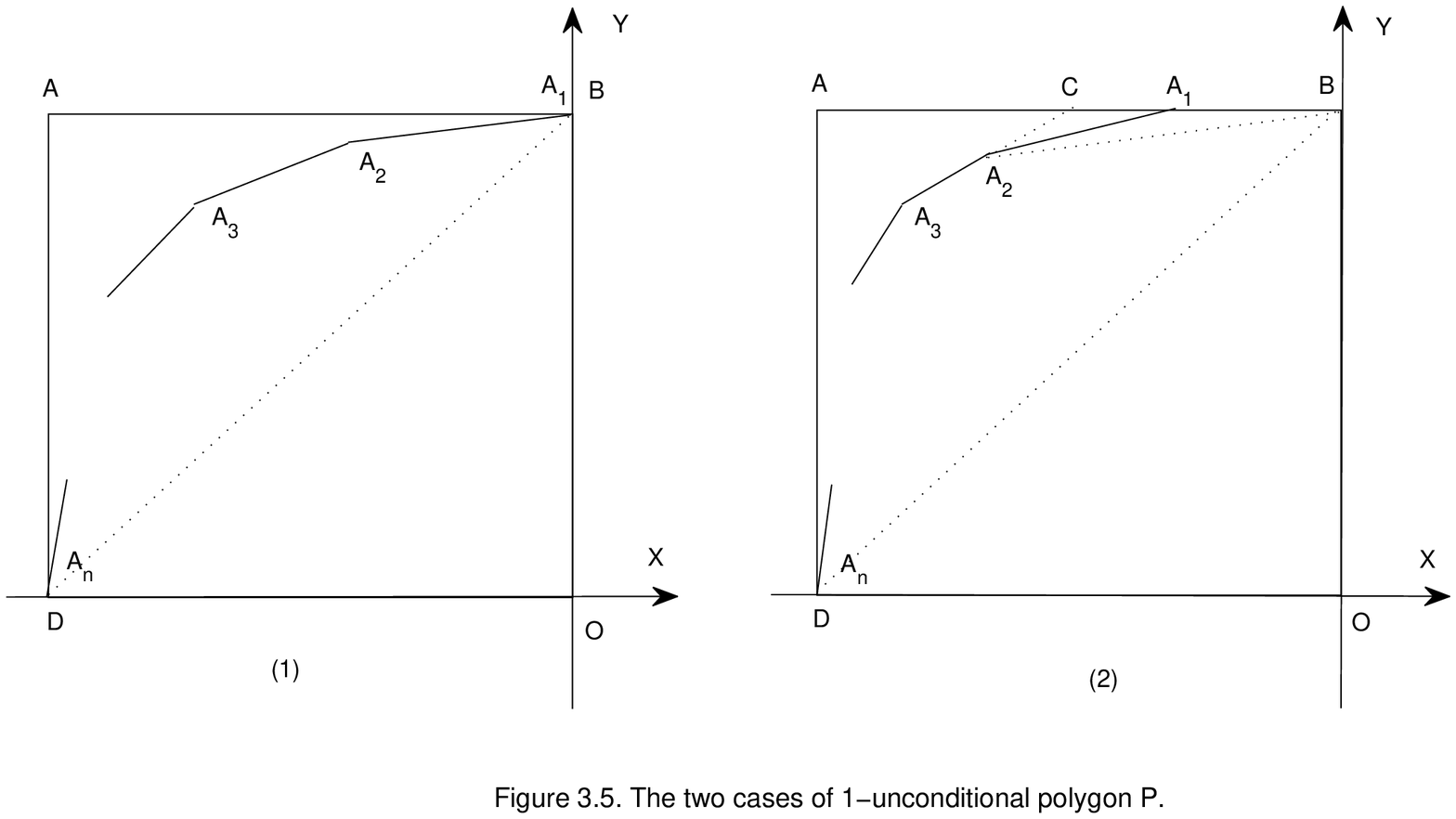 }
\end{figure}
\begin{proof}

Let $A_1, A_2, \cdots, A_n$ be the vertices of $P$ contained in the
domain $\{(x,y): x\leq 0, y\geq 0\}$ and the slopes of lines $OA_i$
($i=1,\cdots n$) are increasing on $i$. Without loss of generality,
suppose that the vertex $A_n$  coincides with point $D$. The vertex
$A_1$ satisfies the following two cases:

(i) $A_1$ coincides with the point $B$;

(ii) $A_1$ does not coincide with the point $B$, but lies on the
line segment $BC$ (C is the point of intersection between two lines
$A_2A_3$ and $AB$).

If $R$ satisfies the case (ii), from the Lemma 3.5, we obtain an
origin-symmetric body of revolution $R_1$ with smaller Mahler volume
than $R$ and its generating domain $P_1$ has fewer vertices than
$P$.

If $R$ satisfies the case (i), then its polar body $R^{\ast}$
satisfies the case (ii). Since $\mathcal {P}(R)=\mathcal
{P}(R^{\ast})$ and $P$ has the same number of vertices as
$P^{\ast}$, from the Lemma 3.5,  we can also obtain an
origin-symmetric body of revolution $R_1$ with smaller Mahler volume
than $R$ and its generating domain $P_1$ has fewer vertices than
$P$.

From the above discuss and the proof of (3.5), let $R_0=R$, we can
get a sequence of origin-symmetric bodies of revolution
$$\{R_0, R_1, R_2\cdots, R_N\},$$ where $N$ is a natural number
depending on the number of vertices of $P$, satisfying $\mathcal
{P}(R_{i+1})\leq \mathcal {P}(R_i)$ ($i=0,1,\cdots, N-1$) and $R_N$
is a cylinder or bicone. Therefore, we have
\begin{eqnarray*}\mathcal {P}(R)\geq \frac{4\pi^2}{3},\end{eqnarray*}
with equality if and only if $R$ is a cylinder or bicone.

\end{proof}

\begin{thm} For any origin-symmetric body of revolution $K$ in $
\mathbb{R}^3$, we have
\begin{eqnarray}
\mathcal {P}(K)\geq \frac{4\pi^2}{3},
\end{eqnarray}
with equality if and only if $K$ is a cylinder or bicone.
\end{thm}
\begin{proof}

By Remark 3, without loss of generality,  suppose that the
generating domain $P$ of $K$ is contained in the square $[-1,1]^2$
and $B, D\in P$.

Since a convex body can be approximated by a polytope in the sense
of the Hausdorff metric (see Theorem 1.8.13 in \cite{Sc93}), hence,
for $P$ and any $\varepsilon>0$, there is a 1-unconditional polygon
$P_{\varepsilon}$ with $\delta(P,P_{\varepsilon})\leq \varepsilon$.
Let $R_{\varepsilon}$ be an origin-symmetric body of revolution
generated by $P_{\varepsilon}$, then $\delta(K, R_{\varepsilon})\leq
\varepsilon.$ Thus, there exists  a sequence of origin-symmetric
bodies of revolution $(R_i)_{i\in \mathbb{N}}$ satisfying
$$\lim _{i\rightarrow \infty}\delta(R_i, K)=0.$$
Since $\mathcal {P}(K)$ is continuous in the sense of the Hausdorff
metric, applying Lemma 3.6, we have

\begin{eqnarray}
\mathcal {P}(K)\geq \frac{4\pi^2}{3},
\end{eqnarray}
with equality if and only if $K$ is a cylinder or bicone.
\end{proof}

%
%
%
%
%
%
%
In the following, we will restate and prove Theorem 1.2 and 1.3.

\begin{thm} Let $f(x)$ be a concave, even and nonnegative
function defined on $[-a,a]$, $a>0$, and for $x^{\prime}\in
[-\frac{1}{a},\frac{1}{a}]$ define
\begin{eqnarray}
f^{\ast}(x^{\prime})=\inf_{x\in[-a,a]}\frac{1-x^{\prime}x}{f(x)}.
\end{eqnarray} Then
\begin{eqnarray}
\left(\int_{-a}^{a} (f(x))^2 dx\right)\left(
\int_{-\frac{1}{a}}^{\frac{1}{a}} (f^{\ast}(x^{\prime}))^2
dx^{\prime}\right)\geq\frac{4}{3},
\end{eqnarray}
with equality if and if $f(x)=f(0)$ or
$f^{\ast}(x^{\prime})=1/f(0)$.
\end{thm}
\begin{proof}
Let $R$ and $R^{\prime}$ be origin-symmetric bodies of revolution
generated by $f(x)$ and $f^{\ast}(x^{\prime})$, respectively, then
their generating domains are
$$D=\{(x,y): -a\leq x\leq a, |y|\leq f(x)\}$$
and
$$D^{\prime}=\{(x^{\prime},y^{\prime}): -\frac{1}{a}\leq x^{\prime}\leq \frac{1}{a},|y^{\prime}|\leq f^{\ast}(x^{\prime})\},$$
respectively.

Next, we prove $D^{\prime}=D^{\ast}$. For
$(x^{\prime},y^{\prime})\in D^{\prime}$ and $(x,y)\in D$, we have
$$
(x^{\prime},y^{\prime})\cdot(x,y) =x^{\prime}x+y^{\prime}y\leq
x^{\prime}x+f^{\ast}(x^{\prime})f(x) \leq
x^{\prime}x+\frac{1-x^{\prime}x}{f(x)}f(x)=1,$$ which implies
$(x^{\prime},y^{\prime})\in D^{\ast}$. If
$(x^{\prime},y^{\prime})\notin D^{\prime}$, then either
$|x^{\prime}|>\frac{1}{a}$ or $|x^{\prime}|\leq\frac{1}{a}$ and
$|y^{\prime}|>f^{\ast}(x^{\prime})$. If $x^{\prime}>\frac{1}{a}$ (or
$x^{\prime}<-\frac{1}{a}$), then for $(a,0)\in D$ (or $(-a,0)\in
D$), we have
$$(x^{\prime},y^{\prime})\cdot(a,0)>1\;\;(\textrm{or}\;\;(x^{\prime},y^{\prime})\cdot(-a,0)>1),$$
which implies $(x^{\prime},y^{\prime})\notin D^{\ast}$. If
$|x^{\prime}|\leq\frac{1}{a}$ and $y^{\prime}>f^{\ast}(x^{\prime})$
(or $y^{\prime}<-f^{\ast}(x^{\prime})$), let
$$f^{\ast}(x^{\prime})=\frac{1-x^{\prime}x_0}{f(x_0)},$$
then for $(x_0, f(x_0))\in D$ (or $(x_0, -f(x_0))\in D$), we have
$$(x^{\prime},y^{\prime})\cdot(x_0, f(x_0))>x^{\prime}x_0+f^{\ast}(x^{\prime})f(x_0)=1$$
$$(\textrm{or}\;\;(x^{\prime},y^{\prime})\cdot(x_0, -f(x_0))>x^{\prime}x_0+f^{\ast}(x^{\prime})f(x_0)=1),$$
which implies $(x^{\prime},y^{\prime})\notin D^{\ast}$. Hence, we
have $D^{\prime}=D^{\ast}$. By Lemma 3.2, we get
$R^{\prime}=R^{\ast}$. By Theorem 3.7, we have
\begin{eqnarray*}
\int_{-a}^{a}(f(x))^2dx
\int_{-\frac{1}{a}}^{\frac{1}{a}}(f^{\ast}(x^{\prime}))^2dx^{\prime}
=\frac{1}{\pi^2}V(R)V(R^{\prime})=\frac{1}{\pi^2}\mathcal {P}(R)
\geq\frac{4}{3},
\end{eqnarray*}
with equality if and if $f(x)=f(0)$ or
$f^{\ast}(x^{\prime})=1/f(0)$.
\end{proof}
By Theorem 3.8, we prove that among parallel sections homothety
bodies in $\mathbb{R}^3$, 3-cubes have the minimal Mahler volume.

\begin{thm}For any parallel sections
homothety body $K$ in $\mathbb{R}^3$, we have
\begin{eqnarray}
\mathcal {P}(K)\geq \frac{4^3}{3!},
\end{eqnarray}
with equality if and only if $K$ is a 3-cube or octahedron.
\end{thm}
\begin{proof}
Let
$$K=\bigcup_{x\in[-a,a]}\{f(x)C+xv\},$$
where $f(x)$ is its generating function and $C$ is homothetic
section. Next, for
$$K^{\prime}=\bigcup_{x^{\prime}\in[-\frac{1}{a},\frac{1}{a}]}\{f^{\ast}(x^{\prime})C^{\ast}+x^{\prime}v\},$$
where $f^{\ast}(x^{\prime})$ is given in (3.62), we prove
$K^{\prime}=K^{\ast}$. For any
$$(x^{\prime},y^{\prime},z^{\prime})\in K^{\prime}\;\;
\textrm{and}\;\; (x,y,z)\in K,$$ we have
$$(0,y^{\prime},z^{\prime})\in f^{\ast}(x^{\prime})C^{\ast}\;\;\textrm{and}\;\;(0,y,z)\in f(x)C.$$
Hence, we have
$$(0,y^{\prime},z^{\prime})\cdot(0,y,z)\leq f^{\ast}(x^{\prime})f(x)\leq \frac{1-x^{\prime}x}{f(x)}f(x)=1-x^{\prime}x.$$
It follows that
$$(x^{\prime},y^{\prime},z^{\prime})\cdot(x,y,z)=x^{\prime}x+(0,y^{\prime},z^{\prime})\cdot(0,y,z)\leq 1,$$
which implies that $(x^{\prime},y^{\prime},z^{\prime})\in K^{\ast}$.

If $(x^{\prime},y^{\prime},z^{\prime})\notin K^{\prime}$, then
either $|x^{\prime}|>\frac{1}{a}$ or $|x^{\prime}|\leq\frac{1}{a}$
and $(0,y^{\prime},z^{\prime})\notin f^{\ast}(x^{\prime})C^{\ast}$.
If $x>\frac{1}{a}$ (or $x<-\frac{1}{a}$ ), then for $(a,0,0)\in K$
(or $(-a,0,0)\in K$), we have
$$(x^{\prime},y^{\prime},z^{\prime})\cdot(a,0,0)>1\;\;(\textrm{or}\;\;(x^{\prime},y^{\prime},z^{\prime})\cdot(-a,0,0)>1),$$
which implies that $(x^{\prime},y^{\prime},z^{\prime})\notin
K^{\ast}$. If $|x^{\prime}|\leq\frac{1}{a}$ and
$(0,y^{\prime},z^{\prime})\notin f^{\ast}(x^{\prime})C^{\ast}$,
there exists $(0,y,z)\in C$ such that
$$(0,y,z)\cdot
(0,y^{\prime},z^{\prime})>f^{\ast}(x^{\prime}).$$ Let
$$f^{\ast}(x^{\prime})=\frac{1-x^{\prime}x_0}{f(x_0)}.$$ For $$(x_0,f(x_0)y,f(x_0)z)\in K$$
 we have
\begin{eqnarray}
&&(x^{\prime},y^{\prime},z^{\prime})\cdot(x_0,f(x_0)y,f(x_0)z)
\nonumber\\
&=&x^{\prime}x_0+f(x_0)(0,y,z)\cdot(0,y^{\prime},z^{\prime})
\nonumber\\
&>&x^{\prime}x_0+f(x_0)f^{\ast}(x^{\prime})
\nonumber\\
&=&x^{\prime}x_0+f(x_0)\frac{1-x^{\prime}x_0}{f(x_0)}
\nonumber\\
&=&1,
\end{eqnarray}
which implies that $(x^{\prime},y^{\prime},z^{\prime})\notin
K^{\ast}$. Hence, we have $K^{\prime}=K^{\ast}$.

Therefore, we obtain
\begin{eqnarray}
\mathcal {P}(K) &=&V(K)V(K^{\prime})
\nonumber\\
&=&\mathcal {P}(C)\int_{-a}^{a}(f(x))^2dx
\int_{-\frac{1}{a}}^{\frac{1}{a}}(f^{\ast}(x^{\prime}))^2dx^{\prime}
\nonumber\\
&\geq& \frac{4^2}{2!}\frac{4}{3}=\frac{4^3}{3!},
\end{eqnarray}
with equality if and only if $K$ is a 3-cube or octahedron.
\end{proof}


\bibliographystyle{amsalpha}

\begin{thebibliography}{99}
\bibitem{Ar04} S. Artstein, B. Klartag, V.D. Milman, {\it On the Santal$\acute{o}$ point of a
function and a functional Santal$\acute{o}$ inequality}, Mathematika
54 (2004), 33-48.
\bibitem{Ba95} K. Ball, {\it Mahler's conjecture and wavelets}, Discrete Comput. Geom.
13 (1995), 271-277.
\bibitem{BM87} J. Bourgain, V. D. Milman, {\it New volume ratio properties for convex symmetric bodies in $\mathbb{R}^n$,}
Invent. Math. 88 (1987), 319-340.
\bibitem{CG06} S. Campi, P. Gronchi, {\it Volume inequalities for
$L_p$-zonotopes,} Mathematika 53 (2006), 71-80.

\bibitem{Ca06} S. Campi and P. Gronchi, {\it On volume product inequalities for convex
sets}, Proc. Amer. Math. Soc. 134 (2006), 2393-2402.

\bibitem{Cam06} S.
Campi and P. Gronchi, {\it Extremal convex sets for
Sylvester-Busemann type functionals}, Appl. Anal. 85 (2006),
129-141.


\bibitem{Fr10} M. Fradelizi, Y. Gordon, M. Meyer, S. Reisner, {\it The case of equality
for an inverse Santal$\acute{o}$ functional inequality}, Adv. Geom.,
10 (2010), 621-630.

\bibitem{Fr07} M. Fradelizi, M. Meyer, {\it Some functional forms of Blaschke-Santal$\acute{o}$
inequality}, Math. Z. 256 (2007), 379-395.

\bibitem{Fr08} M. Fradelizi, M. Meyer, {\it Increasing functions and inverse Santal$\acute{o}$
inequality for unconditional functions}, Positivity 12 (2008),
407-420.

\bibitem{Fra08}M. Fradelizi, M. Meyer, {\it Some functional inverse Santal$\acute{o}$
inequalities}, Adv. Math. 218 (2008), 1430-1452.

%
%


\bibitem{Ga06} R. J. Gardner, {\it Geometric tomography}, Second
edition. Encyclopedia of Mathematics and its Applications, 58.
Cambridge University Press, Cambridge, 2006.






\bibitem{GMR88} Y. Gordon, M. Meyer and S. Reisner, {\it Zonoids with minimal volume-product--a new
proof,} Proc. Amer. Math. Soc. 104 (1988), 273-276.


\bibitem{Ki}  J. Kim, S. Reisner,
{\it Local minimality of the volume-product at the simplex},
Mathematika, in
 press.

\bibitem{Ku08} G. Kuperberg, {\it From the Mahler Conjecture to Gauss Linking Integrals,} Geom. Funct. Anal. 18 (2008), 870-892.

\bibitem{LR98} M. A. Lopez, S. Reisner, {\it A Special Case of Mahler's
Conjecture}, Discrete Comput. Geom. 20 (1998), 163-177.
%


\bibitem{LYZ10} E. Lutwak, D. Yang and G. Zhang, {\it A volume inequality for polar
bodies}, J. Differential Geom. 84 (2010) 163-178.



\bibitem{LL10} Y. Lin, G. Leng, {\it Convex bodies with minimal
volume product in $\mathbb{R}^2$--a new proof}, Discrete Math. 310
(2010), 3018-3025.

\bibitem{Ma39} K. Mahler, {\it Ein $\ddot{U}$bertragungsprinzip f$\ddot{u}$r konvexe K$\ddot{o}$rper}, Casopis Pest. Mat. Fys. 68
(1939), 93-102.
\bibitem{Ma39b} K. Mahler, {\it Ein Minimalproblem f$\ddot{u}$r konvexe Polygone}, Mathematica (Zutphen) B.
7 (1939), 118-127.


\bibitem{Me86} M. Meyer, {\it Une caract$\acute{e}$risation volumique de certains espac\'{e}s normes de dimension finie.} Israel J. Math. 55 (1986), 317-326.


\bibitem{Me91} M. Meyer, {\it Convex bodies with minimal volume product in $\mathbb{R}^2$}, Monatsh. Math. 112 (1991),
297-301.

\bibitem{Me98} M. Meyer and S. Reisner, {\it Inequalities involving integrals of Polar-conjugate concave
functions}, Monatsh. Math. 125  (1998), 219-227.

\bibitem{MR06} M. Meyer and S. Reisner, {\it Shadow systems and
volumes of polar convex bodies}, Mathematika 53(2006), 129-148.



\bibitem{NPRZ09} F. Nazarov, F. Petrov, D. Ryabogin and A. Zvavitch,
{\it A remark on the Mahler conjecture: local minimality of the unit
cube}, Duke Math. J. 154 (2010), 419-430.

%
%

%

\bibitem{Re85}  S. Reisner, {\it Random polytopes and the volume-product of symmetric convex bodies}, Math.
Scand. 57 (1985), 386-392.

\bibitem{Re86}  S. Reisner, {\it Zonoids with minimal volume-product}, Math. Z. 192 (1986),
339-346.

\bibitem{Re87}  S. Reisner, {\it Minimal volume product in Banach spaces with a 1-unconditional basis}, J.
Lond. Math. Soc. 36 (1987), 126-136.

\bibitem{SR81} J. Saint Raymond, {\it Sur le volume des corps convexes sym etriques}, Seminaire d'initiation ¡®al¡¯ Analyse, 1980/1981, Publ. Math. Univ. Pierre et Marie Curie,
Paris, 1981, 1-25.

%

\bibitem{Sc93} R. Schneider, {\it Convex Bodies: The Brunn-Minkowski
Theory}, Encyclopedia Math. Appl., vol. 44, Cambridge University
Press, Cambridge, 1993.

\bibitem{Ta07} T. Tao, {\it Structure and Randomness: pages from year one of a mathematical
blog}, Amer. Math. Soc. (2008), 216-219.





\end{thebibliography}

\end{document}